\documentclass[leqno,12pt]{amsart} 
\title[Graded contractions on orthogonal Lie algebras]{Graded contractions on the \\orthogonal Lie algebras of dimensions 7 and 8}

\author[]{Cristina Draper$^{\star}$, Thomas Leenen Meyer$^{\ast}$, Juana S\'anchez-Ortega$^{\bullet}$
}

\usepackage{amssymb,color, bm}

\thanks{${}^*$ Supported by Junta de Andaluc\'{\i}a  through project  FQM-336,   and  by the Spanish Ministerio de Ciencia e Innovaci\'on   through projects PID2020-118452GB-I00 and PID2023-152673GB-I00, with FEDER funds.}

\usepackage{amssymb, mathrsfs, amsmath, tabu, amsthm}
\usepackage{hyperref}
\usepackage{tikz-cd, graphicx} 
\usepackage{pdfpages}
\usepackage{enumitem}
\usepackage{ragged2e}
\usetikzlibrary{decorations.markings}

\hypersetup{
	colorlinks = true,
	urlcolor = blue,
	linkcolor = blue,
	citecolor = red
}
\newcommand{\gla}{\mathfrak{gl}} 
\newcommand{\la}{\mathfrak{L}} 

\newcommand{\f}[1]{\mathfrak{#1}}
\newcommand{\supp}[1]{\Supp^{#1}} 

\newcommand{\bb}[1]{\mathbb{#1}}

\newcommand{\ep}{\varepsilon}

\newcommand{\lae}{\la^{\ep}}
\newcommand{\weyl}{\mathcal{W}}

\newcommand{\comment}[1]{}

\DeclareMathOperator{\Supp}{\mathcal{S}}
\DeclareMathOperator{\Der}{Der}
\DeclareMathOperator{\Aut}{Aut}
\DeclareMathOperator{\Stab}{Stab}

\DeclareMathOperator{\ad}{ad}
\DeclareMathOperator{\id}{id}
\DeclareMathOperator{\Spec}{Spec}

\theoremstyle{plain} 
\newtheorem{theorem}{Theorem}[section]
\newtheorem{lemma}[theorem]{Lemma}
\newtheorem{cor}[theorem]{Corollary}
\newtheorem{prop}[theorem]{Proposition}

\theoremstyle{definition} 
\newtheorem{define}[theorem]{Definition}

\newtheorem{remark}[theorem]{Remark}
\newtheorem{example}[theorem]{Example}

\newtheorem{conclusion}[theorem]{Conclusion}

\begin{document}


\maketitle

\begin{abstract}
Graded contractions of certain non-toral  $\mathbb{Z}_2^3$-gradings on the  simple Lie algebras $\mathfrak{so}(7,\mathbb C)$ and $\f{so}(8,\mathbb C)$ are classified up to two notions of equivalence. In particular, there arise two large families of Lie algebras (the majority of which are solvable) of dimensions 21 and 28. This is achieved as a significant generalization of the classification of related graded contractions on $\mathfrak g_2$, the derivation algebra of the octonion algebra. Many of the results can be further extended to any \emph{good} $\mathbb Z_2^3$-grading on an arbitrary Lie algebra.

\end{abstract}
	
 \section{Introduction and some preliminaries}  
 
  Some constructions of Lie algebras modifying a concrete Lie algebra are interesting for Physics: the most studied  of these constructions are contractions, degenerations and deformations (see, for instance, \cite{Fialowski2}). The first two in general produce Lie algebras more abelian than the original one, while the third procedure, the deformation, makes Lie algebras with more complicated   Lie bracket. In some sense, these procedures are mutually opposite.
 Also these processes are relevant in Mathematics, producing new types of Lie algebras, which is always important in  pursuit of the list of Lie algebras of a given dimension, up to isomorphism.

There are many definitions of contractions, according to the approaches and different goals of the different researchers.
 Perhaps graded contractions are the most recent and the least studied. They were introduced by some physicists in \cite{CPSW},
 consisting essentially in preserving Lie gradings through the contraction process. This is not odd, any Lie algebra admits many gradings and the gradings provide a lot of information on the structure of the algebra.
 In a graded contraction  homogeneous components are preserved, in particular  this offers a collection of Lie algebras graded by the same group. 
 Some works worth mentioning   on  graded contractions are \cite{ref33,ref34}, by Weimar-Woods.
 There has been a little more activity in terms of examples of graded contractions of a given Lie algebra, although not too much either.
 For instance, graded contractions on the 8-dimensional $sl(3;\mathbb C)$, which 
 admits four different fundamental gradings,
 are considered in \cite{CPSW,checos,gr-cont,checos06}. Also graded contractions of affine Lie algebras are considered  in 
 \cite{affine1,affine2} with the purpose of getting new Lie algebras of infinite dimension, where the examples developed are  $\hat A_1$ and $\hat A_2$ ($\hat{\mathfrak g}=\mathfrak g\otimes\mathbb C[t,t^{-1}]\oplus\mathbb C k$, for $\mathbb C[t,t^{-1}]$   the associative algebra of
the Laurent polynomials in $t$, and $k$  a central extension). A program for the computation of graded contractions of Lie algebras and their representations for a given grading group is presented in \cite{conordenador}. 

The complexity of a general treatment for getting graded contractions can be observed in the aforementioned works. 
Nevertheless, we propose to make a leap both in the Lie algebra dimension and in the grading group. 
We will build on our previous work \cite{Thomas1} on graded contractions of the simple Lie algebra $\mathfrak g_2$ over $\bb C$. This
  is the Lie algebra of derivations of the complex  octonion algebra, and is 
endowed with a   
 $\bb{Z}_2^3$-grading $\Gamma_{\f{g}_2}$ induced from the natural     
 $\bb{Z}_2^3$-grading on the   octonion algebra. 
The resulting $\mathbb Z_2^3$-grading on $\mathfrak g_2$ is highly symmetric, and it shares  some key features with certain $\mathbb Z_2^3$-gradings of $\mathfrak b_3$ and $\mathfrak d_4$, which are those that will make it possible to transfer many of the results from  \cite{Thomas1}.
This paper  is going to provide us with quite strong results, which include complete classifications of graded contractions of $\Gamma_{\f{g}_2}$ up to isomorphism, at least for the field of complex numbers. A combinatorial approach 
 made it possible to obtain these classifications without using a computer.

Let $\Gamma: \la = \bigoplus_{g\in G} \la_g$ be a $G=\bb{Z}_2^3$-grading on a complex Lie algebra $\la$, that is a vector space decomposition such that $[\la_g,\la_h] \subseteq \la_{g + h}$, for all $g, \, h$ in $G$.  
A \emph{graded contraction} of $\Gamma$ is a map $\varepsilon\colon G\times G\to \bb{C}$ such that the vector space $\la$ endowed with product $[x, y]^\varepsilon := \varepsilon(g, h)[x, y]$, for $x\in \la_g,  y\in \la_h$, is a Lie algebra. We write $\la^\ep$ to refer to $(\la, [\cdot, \cdot]^\epsilon)$,
which is $G$-graded too with $(\la^\ep)_g=\la_g$. We will say that two graded contractions $\ep$ and $\ep'$ are \emph{equivalent} (respectively, \emph{strongly equivalent}) if the corresponding graded algebras are isomorphic as   graded algebras (respectively graded-isomorphic), we mean, 
there exists $f\colon\la^\ep\to\la^{\ep'}$   isomorphism of  Lie algebras such that for any $g\in G$ there is $h \in G$ with $f(\la_g) = \la_h$
(respectively, $f(\la_g) = \la_g $ for any $g\in G$). We write $\ep\sim\ep'$ (respectively $\ep \approx \ep'$). Clearly $\ep \approx \ep'\Rightarrow  \ep \sim \ep'$.
 
  Remark~2.4 in \cite{Thomas1} describes when an arbitrary map $\ep\colon G\times G\to \bb{C}$ is a graded contraction of $\Gamma$.
Denote also by $\ep\colon G \times G \times G \to \bb{C}$ the ternary map defined as
$\ep(g, h, k) := \ep(g, h + k)\ep(h, k)$. 
The conditions for  $\ep$ to be  a graded contraction of $\Gamma$  are:
\begin{enumerate}
\item[(a1)] $\big(\ep(g, h) - \ep(h, g)\big)[x, y] = 0,$
\item[(a2)] $\big(\ep(g, h, k) - \ep(k, g, h)\big)[x, [y, z]]
+ \big(\ep(h, k, g) - \ep(k, g, h)\big)[y, [z, x]] = 0$,
\end{enumerate}
 for all $g, h, k \in G$ and any choice of homogeneous elements $x\in \la_g, \, y\in \la_h, \, z\in\la_k$. 
 These conditions appear to be strongly dependent on   the considered grading on $\la$. Thus,  a map $\ep\colon G\times G\to \bb{C}$ can be a graded contraction of a $G$-graded Lie algebra $\la$ but not of a different $G$-graded Lie algebra $\la'$. 
  For instance, if $\la$ is abelian, any vector space decomposition $\Gamma$ of $\la$ labelled in $G$ is in fact a $G$-grading, and an arbitrary 
   map $\ep\colon G\times G\to \bb{C}$ turns out to be a graded contraction of $\Gamma$. Clearly, this example is not very interesting since $\la^\ep$ and $\la$ are just the same algebra. However,  it is illustrative of the fact that the concrete grading is involved, so that we cannot think of graded contractions on a Lie algebra without choosing a particular $G$-grading.

 The problem of determining the equivalence classes of the graded contractions of  a $G$-grading $\Gamma$ on a  Lie algebra $\la$ is in general a difficult task which, as mentioned,   strongly depends on $\Gamma$. But for 
     $\Gamma_{\f{g}_2}$ this equivalence problem has been completely solved in \cite[Theorem~4.13]{Thomas1}.  
     The current article  takes advantage of those results  
     to achieve other classifications of graded contractions, seemingly more difficult because the algebras  which are contracted have higher dimension. Namely, they are $\mathfrak b_3$ and $\mathfrak d_4$, of dimensions 21 and 28. 
     As a byproduct, our procedure will provide new examples of  non-simple and non-abelian Lie algebras of considerably large dimensions. 
     Obviously the classification of Lie algebras of dimension 21  is not even remotely known  (nor that of dimension 28 either) so having some
     non-trivial examples is convenient.
       We rely on the following   observation: the mentioned orthogonal  Lie algebras possess $\bb Z_2^3$-gradings with similar properties to $\Gamma_{\f{g}_2}$.
       For  instance with the same symmetry group, and hopefully with the same graded contractions!  
        It seems surprising but it will turn out that the graded contractions are the \emph{same} (Theorem~\ref{th_lasclasesdeverdad}). The same maps but of course not the same Lie algebras: we will obtain a new list of Lie algebras of dimensions 21 and 28, whose properties are described in Theorem~\ref{teo_lasalgebras}. \smallskip

      This manuscript is structured as follows. Section~\ref{se_propiedadesb3d4} deals with the orthogonal Lie algebras $\f{so}(8,\bb C)$, of type $D_4$, and $\f{so}(7,\bb C)$, of type $B_3$, and describes a \emph{good} $\bb Z_2^3$-grading on each of them. The
      chosen  gradings fit well with the chain $\f{g}_2\subset \f{so}(7,\bb C)\subset \f{so}(8,\bb C)$.
       We will also prove some technical properties on the brackets we will need afterwards. 
      Next in    Section~\ref{se_resultadosengrcont}, we shall define good $\bb Z_2^3$-gradings, which are exactly those whereupon the results in \cite{Thomas1} can be applied. If a grading is good, for each graded contraction   there is another one which is admissible and equivalent to the original one.  For admissible graded contractions the problem can be transferred to a combinatorial context,   where the Lie bracket is no longer a fundamental piece. The  supports of such admissible 
      graded contractions turn out to be nice sets, certain combinatorial objects  defined and classified in \cite[\S3.3]{Thomas1}. 
      For most of the nice sets, there is only one equivalence class of admissible graded contractions with that  particular support 
      (it is in the proof of uniqueness that the algebraically closed field is used).
       The  remaining three nice sets are related to   infinite families of Lie algebras. 
       We will use the high symmetry of our gradings (to be precise, the fact that their Weyl group is as  large as possible)
       to prove that collinear nice sets give isomorphic Lie algebras. 
       The nice sets up to collineation, jointly with the properties of the algebras attached to those supports, are described in  Theorem~\ref{teo_lasalgebras}. However, the question 
       of the classification of the  isomorphism classes is very subtle.
       A first conflictive point is Theorem~\ref{teo_fuerte}, where we prove that  if there exists a graded-isomorphism between $\la^\ep$ and $\la^{\ep'}$, then there exists another one that acts scalarly on each homogeneous component (this translates the equivalence by normalization). This is not combinatorial at all, we have to dive into the concrete structure of the graded algebra we  wish  to contract. 
       The second difficult point is to relate the classification of the  isomorphism classes with the classification of the strong isomorphism classes.  It
        is not true 
       that an admissible graded contraction equivalent to another one, $\ep$, is necessarily strongly equivalent to $\sigma\cdot \ep$ for some collineation $\sigma$, but, according to Proposition~\ref{nueva}, this can be assured if the supports are collinear, which is  usually the situation. From here,  the  main result, Theorem~\ref{th_lasclasesdeverdad}, follows. Conclusion~\ref{conclu} tries to  summarize the main results. 
       Observe that an effort has been made not only to adapt the results of \cite{Thomas1}, but also to simplify them. 
       We have tried to be clear but brief.
       Observe also that the analogy between \cite[Theorem 4.13]{Thomas1} and 
       Theorem~\ref{th_lasclasesdeverdad} does not mean that the adaptation is trivial.

      Throughout this paper, the considered ground  field will be the one of complex numbers, because that was the hypothesis in \cite{Thomas1}. 
    But   Section~\ref{se_real} will be devoted to give brief comments on the real case, highlighting which of the results will remain applicable to this context as well.
     
     
     \section{ $\bb Z_2^3$-gradings on $\mathfrak b_3$ and $\mathfrak d_4$ induced from the octonion algebra} \label{se_propiedadesb3d4}

     The \emph{complex octonion algebra} $\mathcal{O}$ has a basis $\{e_0,\dots, e_7\}$ with  an identity element $e_0=1$, $e_i^2=-1$ for all $i\ne0$,  and $e_ie_j=-e_je_i=e_k$, $e_je_k=-e_je_k=e_i$,
$e_ke_i=-e_ke_i=e_j$ whenever 
$$(i,j,k)\in\mathbf{L}:=\{(1,2,5),(5,6,7), (7,4,1),(1,3,6),(6,4,2),(2,7,3),(3,4,5)\}.$$
That is, the product is given  by following the arrows of the Fano plane:
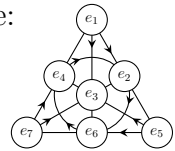
\vspace{-15pt}\begin{flushright}
\begin{tikzpicture}[scale=0.5, every node/.style={transform shape}] \label{fano}
\draw [postaction={decoration={markings, mark= at position 0.75 with {\arrowreversed{stealth}}}, decorate}] 
(30:1) -- (210:2);
\draw [postaction={decoration={markings, mark= at position 0.75 with {\arrowreversed{stealth}}}, decorate}]
(150:1) -- (330:2);
\draw [postaction={decoration={markings, mark= at position 0.75 with {\arrowreversed{stealth}}}, decorate}]
(270:1) -- (90:2);
\draw [postaction={decoration={markings, mark= at position .24 with {\arrowreversed{stealth}}}, decoration={mark= at position .57 with {\arrowreversed{stealth}}},decoration={mark= at position .9 with {\arrowreversed{stealth}}}, decorate}]
(90:2)  -- (210:2) -- (330:2) -- cycle;
\draw [postaction={decoration={markings, mark= at position .3 with {\arrowreversed{stealth}}}, decoration={mark= at position .63 with {\arrowreversed{stealth}}},decoration={mark= at position .96 with {\arrowreversed{stealth}}}, decorate}] (0:0)  circle (1);
\draw 
(30:1) node[circle, draw, fill=white]{{$e_2$}} 
(210:2) node[circle, draw, fill=white]{{$e_7$}} 
(150:1) node[circle, draw, fill=white]{{$e_4$}}
(330:2) node[circle, draw, fill=white]{{$e_5$}} 
(270:1) node[circle, draw, fill=white]{{$e_6$}} 
(90:2) node[circle, draw, fill=white]{{$e_1$}} 
(0:0) node[circle, draw, fill=white]{{$e_3$}}; 
\end{tikzpicture}
 %
\quad\quad\quad\ \end{flushright} 
For our convenience, for any $i,j\in I:=\{1,\dots,7\}$, $i\ne j$, we will denote by $i*j\in I$ the only distinct index in the same line as $i$ and $j$. 
And we denote the line by $\ell_{ij}:=\{i,j,i*j\}$, without taking the order into account. 
We will say that $i\prec j$ if $e_ie_j=e_{i*j}$, so that  $j\prec i$ if $e_ie_j=-e_{i*j}$. 
 
 Let us emphasize the well-known fact that $\mathcal{O}$ is a non-associative algebra: if $(x,y,z)=(xy)z-x(yz)$ denotes the associator, we have $(e_1,e_2,e_3)=2e_4\ne0$. But $\mathcal{O}$ is alternative: the associator alternates, $(x,x,y)=(y,x,x)=0$, which says that $\mathcal{O}$ is closely to be associative.
  The quadratic form $n\colon \mathcal{O}\to\mathbb C$, $n(\sum_ix_ie_i)=\sum x_i^2$,  usually called the \emph{norm},	 is non-degenerate and multiplicative ($n(xy) = n(x)n(y)$).  
 The linear map $t\colon \mathcal{O}\to\mathbb C$, $t(\sum_ix_ie_i)=2x_0$, is the \emph{trace}, and the set of zero trace elements is denoted by $\mathcal{O}_0$, spanned by $\{e_1,\dots, e_7\}$. The restriction $n\vert_{\mathcal{O}_0}$ is non-degenerate too. We also use $n\colon \mathcal{O}\times\mathcal{O}\to\bb C$ to denote the polar form. 
 
 Some simple Lie algebras appear in direct connection to octonions, namely,
  \begin{itemize}
 \item $\mathfrak{so}(\mathcal{O},n)=\{f\in  \gla(\mathcal{O}):n(f(x),y)+n(x,f(y))=0\ \forall x,y\in\mathcal{O}\}\cong \f{so}(8,\bb C),$ which is simple of type $D_4$,
 \item $\mathfrak{so}(\mathcal{O}_0,n\vert_{\mathcal{O}_0})\cong \mathfrak{so}(7,\bb C)$, which is simple of type $B_3$, 
  \item $\Der(\mathcal{O}) =\left \{d\in\mathfrak{gl}(\mathcal{O})\mid  d(xy)=d(x)y+xd(y), \ \forall \, x,y\in \mathcal{O}\right\}$,   simple of type $G_2$.
 \end{itemize}
 These algebras are chained as follows
  $$
 \mathfrak{g}_2:= \Der(\mathcal{O}) \subset \mathfrak{b}_3:=\{f\in \mathfrak{so}(\mathcal{O},n):f(1)=0\}\subset\mathfrak{d}_4:= \mathfrak{so}(\mathcal{O},n).
  $$
Indeed, if $f\in \mathfrak{so}(\mathcal{O},n)$ with $f(1)=0$, then $n(1,f(y))=0$ for all $y\in\mathcal{O}$, so that $f(\mathcal{O})\subset 1^\perp=\mathcal{O}_0$ and  the restriction map $f\mapsto f\vert_{\mathcal{O}_0} $ gives a natural isomorphism between 
    $ \{f\in \mathfrak{so}(\mathcal{O},n):f(1)=0\}\le \mathfrak{so}(\mathcal{O},n)$ and $\mathfrak{so}(\mathcal{O}_0,n)$, of $B_3$-type. 
    Also, any derivation $d\in\Der(\mathcal{O})$ satisfies $d(1)=0$ and is skew-symmetric relative to the norm.   We can be even more precise in the description of the elements, with the help of the left and right   multiplication operators $L_x,R_x\colon \mathcal{O}\to \mathcal{O}$ given by $L_x(y)=xy$  and $R_x(y)=yx$. If we denote by $f^*$ the adjoint map of $f\in  \gla(\mathcal{O})$ with respect to the norm, that is, the linear map determined by $n(f(x),y)=n(x,f^*(y))$, then $L_x^*=L_{\bar x}$ and $R_x^*=R_{\bar x}$,
    for the involution $\bar x=t(x)-x$. In particular, $L_x,R_x\in\mathfrak{so}(\mathcal{O},n)=\{f\in  \gla(\mathcal{O}):f^*=-f\}$ if $x \in \mathcal{O}_0$ ($\bar x=-x$),
    and the derivation $D_{x,y}=[L_x,L_y]+[L_x,R_y]+[R_x,R_y]$ also belongs to $\mathfrak{so}(\mathcal{O},n)$. Moreover,
    $\Der(\mathcal{O})=\{\sum_iD_{x_i,y_i}:x_i,y_i\in\mathcal{O}\}\subset \{f\in \mathfrak{so}(\mathcal{O},n):f(1)=0\}$, since
    any derivation $d$ fulfills $d(1)=2d(1)=0$. 
    Note that if the map $\ad x:=L_x-R_x$ is  a derivation, then $0=[x,y]z+y[x,z]-[x,yz]=3(x,y,z)$, so that $(x, \mathcal{O}, \mathcal{O})=0$  and necessarily $x\in  \bb C 1$.
    This means that $\Der(\mathcal{O}) \cap\{\ad x:x\in\mathcal{O}_0\}=0$, and, by dimension count,
    $$
    \mathfrak{b}_3=\Der(\mathcal{O}) \oplus\{\ad x:x\in\mathcal{O}_0\}.
    $$
    Similarly, $(L_x+R_x)(1)=2x$, so that, again by dimension count,
\begin{equation}\label{eq_d4}
    \mathfrak{d}_4=\mathfrak{b}_3 \oplus\{L_x+R_x:x\in\mathcal{O}_0\}.
\end{equation}
    From here, $\mathfrak{d}_4=\Der(\mathcal{O}) \oplus L_{\mathcal{O}_0} \oplus R_{\mathcal{O}_0}$, which is  moreover a decomposition of $\mathfrak{d}_4$ as a sum of irreducible $\mathfrak{g}_2$-modules.

Recall that, for $G$ an abelian group,  a $G$-grading on an algebra $A$ is a decomposition $\Gamma:A=\oplus_{g\in G}A_g$ compatible with the product, that is, $A_gA_h\subset A_{g+h}$.
 The octonion algebra is $\bb Z_2^3$-graded in a natural way: we declare $e_i\in\mathcal{O}_{g_i}$ for 
 \[
\begin{array}{cccc}
g_0 := (0, 0, 0), & \quad g_1 :=(1, 0, 0), & \quad g_2 := (0, 1, 0), & \quad g_3 :=(0, 0, 1),
\\
g_4 := (1, 1, 1), & \quad g_5:= (1, 1, 0), & \quad g_6:= (1, 0, 1), & \quad g_7 := (0, 1, 1).
\end{array}
\]
Indeed, for $i,j\in I$, $i\ne j$, clearly $\mathcal{O}_{g_i}^2\subset \mathcal{O}_{g_0}$, from $e_i^2=-1$, and  also $\mathcal{O}_{g_i}\mathcal{O}_{g_j}\subset \mathcal{O}_{g_i+g_j}$, since
$e_ie_j\in\{\pm e_{i*j}\}$ and taking into account that our labelling has been chosen to satisfy $g_i+g_j=g_{i*j}$.
In fact, this $\bb Z_2^3$-grading  
 is obtained by  thrice applying a Cayley-Dickson doubling process starting from the field, which introduces a factor $\bb Z_2$ in each step.
 We call $\Gamma_{\mathcal{O}}$ this grading   on $\mathcal{O}$, which  induces a $G=\bb Z_2^3$-grading on $\gla(\mathcal{O})=\oplus_{g\in G} \,\gla(\mathcal{O})_g$ given by 
\begin{equation}\label{eq_gradgl}
 \gla(\mathcal{O})_g=\{f\colon \mathcal{O}\to \mathcal{O}\ \mathrm{ linear}:f( \mathcal{O}_h)\subset \mathcal{O}_{g+h}\ \forall h\in G\}.
\end{equation}
 (Any decomposition of a vector space into direct sum of subspaces labelled over a group $G$  induces such $G$-grading on the algebra $\gla(V)$,  which follows from considering block matrices.)
 It turns out that any of the three Lie subalgebras 
$\la\in\{ \mathfrak{d}_4, \, \mathfrak{b}_3,\, \mathfrak{g}_2\}$
 is a graded subspace of $\gla(\mathcal{O})$, that is, 
writing $\la_g=\gla(\mathcal{O})_g\cap \la$ for each $g\in G$, then $\la=\oplus_{g\in G}\la_g$. 
In case $\la= \mathfrak{d}_4$, the key is that $\{e_i\}_{i=0}^7$ is an orthonormal basis of $\mathcal{O}$; so that $n(\mathcal{O}_g,\mathcal{O}_h)=0$ if $g+h\ne e$.
 Thus,
if $f\in\mathfrak{so}(\mathcal{O},n)$ can be written as $f=\sum f_g$ with $f_g(\mathcal{O}_h)\subset \mathcal{O}_{g+h}$, we can  check that
$n(f_g(x),y)+n(x,f_g(y))=0$ for any $x,y\in\mathcal{O}$. Indeed, assume without loss of generality that $x,y$ are homogeneous elements, $x\in\mathcal{O}_h$ and $y\in\mathcal{O}_k$. 
If $g+h+k\ne e$, then $n(f_g(x),y)=0=n(x,f_g(y))$, and the result holds. Otherwise $ n(f_g(x),y)=n(\sum_\sigma f_\sigma(x),y)=n(f(x),y)$, since $n(f_\sigma(x),y)=0$ whenever $\sigma\ne g=h+k$. Similarly $n(x,f_g(y))=n(x,f(y))$ and so $n(f_g(x),y)+n(x,f_g(y))=n(f(x),y)+n(x,f(y))=0$.
This means that $f_g\in\mathfrak{so}(\mathcal{O},n)$ and $  \mathfrak{d}_4$ is a graded subspace.
The proof works identically  for $\la= \mathfrak{b}_3$.  
In any of the three cases, the
 $ \bb Z_2^3$-grading  on  $\la$ inherited from $\gla(\mathcal{O})$ will be denoted as $\Gamma_{\la}$. 
 
 This paper is devoted to explore the graded contractions of $\Gamma_{ \mathfrak{b}_3}$ and $\Gamma_{ \mathfrak{d}_4}$, taking advantage 
  of the  many   similarities between these two gradings and $\Gamma_{ \mathfrak{g}_2}$, which was studied in \cite{Thomas1}. 
  The three of them  share a surprising property: the grading $\Gamma_{\la}$ is a decomposition of the algebra $\la$ into direct sum of seven undistinguishable Cartan subalgebras such that the product of two of them fulfills another one. In particular, any basis formed by homogeneous elements is formed entirely by semisimple elements. All this is  well-known for  $\Gamma_{\mathfrak{g}_2}$,  for instance from \cite{GradsG2,EK13}. 
We will prove in the next lemmas this and other properties  we will need through this work, for $\la\in\{\f{d}_4,\f{b}_3\}$.

Recall that,  for any $x,y\in\mathcal{O}$,   the skew-symmetric maps 
$$\varphi_{x,y}:=n(x,-)y-n(y,-)x\in\mathfrak{so}(\mathcal{O},n),$$
permit to describe the whole orthogonal Lie algebra as $\f{so}(\mathcal{O},n)=\{\sum \varphi_{x_i,y_i}:x_i,y_i\in\mathcal{O}\}=\varphi_{\mathcal{O},\mathcal{O}}$. The bracket of arbitrary elements is very easy to
manage from the identity
  $[\psi, \varphi_{x,y}]=\varphi_{\psi(x),y}+\varphi_{x,\psi(y)}$, which holds for any $\psi\in\mathfrak{so}(\mathcal{O},n)$ and any $x,y\in\mathcal{O}$.
  Similarly,  $\f{b}_3= \varphi_{\mathcal{O}_0,\mathcal{O}_0}$.

 \begin{lemma}\label{le_corchete}
The homogeneous components of the gradings  $\Gamma_{ \mathfrak{d}_4}$ and $\Gamma_{ \mathfrak{b}_3}$ are, respectively,
$$
\begin{array}{ll}
( \mathfrak{d}_4)_g=\f{so}(\mathcal{O},n)_g=\Der(\mathcal{O})_{g}\oplus L_{\mathcal{O}_{g}}\oplus R_{\mathcal{O}_{g}}=\sum_{h\in G}\varphi_{\mathcal{O}_h,\mathcal{O}_{g+h}},\qquad\qquad&(\mathfrak{d}_4)_e= 0;\\
( \mathfrak{b}_3)_g=\Der(\mathcal{O})_{g}\oplus \{\ad x:x\in\mathcal{O}_g\}=\sum_{h\in G\setminus\{e,g\}}\varphi_{\mathcal{O}_h,\mathcal{O}_{g+h}},& ( \mathfrak{b}_3)_e=0;
\end{array}
$$
for any $g\ne e$, $g\in G=\bb Z_2^3$.   In particular the nonzero homogeneous components are Cartan subalgebras. 
 Furthermore,
\begin{itemize}
 \item[\rm(i)] For any $i\in I$, choose $j\ne i$ and $k\notin\ell_{ij}$. Then  the sets 
 $$
 \mathcal{B}_{ijk}=\{\varphi_{e_{0},e_{i}},\varphi_{e_{j},e_{i*j}},\varphi_{e_{k},e_{i*k}},\varphi_{e_{j*k},e_{i*j*k}}\}\  \mathrm{ and }\ 
 \mathcal{B}'_{ijk}=\{ \varphi_{e_{j},e_{i*j}},\varphi_{e_{k},e_{i*k}},\varphi_{e_{j*k},e_{i*j*k}}\}
 $$
 are basis of the homogeneous components $( \mathfrak{d}_4)_{g_i}$ and $( \mathfrak{b}_3)_{g_i}$, respectively.
  \item[\rm(ii)] For   arbitrary elements  
  \begin{equation}\label{eq_zyz'}
  \begin{array}{l}
  z=d\varphi_{e_{0},e_{i}}+a\varphi_{e_{j},e_{i*j}}+b\varphi_{e_{k},e_{i*k}}+c\varphi_{e_{j*k},e_{i*j*k}}\in ( \mathfrak{d}_4)_{g_i},\\
  z'=a\varphi_{e_{j},e_{i*j}}+b\varphi_{e_{k},e_{i*k}}+c\varphi_{e_{j*k},e_{i*j*k}}\in ( \mathfrak{b}_3)_{g_i},\qquad \qquad (a,b,c,d\in\bb C)
  \end{array}
  \end{equation}
  the matrices of the endomorphisms $\ad^2z\vert_{ (\mathfrak{d}_4)_{g_j}}$ relative to $\mathcal{B}_{jik}$ 
  and $\ad^2z'\vert_{ (\mathfrak{b}_3)_{g_j}}$ relative to $\mathcal{B}'_{jik}$ are, respectively,
  \scriptsize
  $$
  \begin{pmatrix} -d^2-a^2 & 2ad&0&0\\2ad&-d^2-a^2 &0&0\\0&0&-b^2-c^2 &2bc\\0&0&2bc&-b^2-c^2 \end{pmatrix}
  \quad\mathrm{and}\quad
   \begin{pmatrix} -a^2 & 0&0\\0&-b^2-c^2 &2bc\\0&2bc&-b^2-c^2 \end{pmatrix}
  .
  $$
  \normalsize 
  In particular the spectrum of $\ad^2z\vert_{ (\mathfrak{d}_4)_{g_j}}$  becomes $\{-(a\pm d)^2, -(b\pm c)^2\}$ 
  (respectively the spectrum of $\ad^2z'\vert_{ (\mathfrak{b}_3)_{g_j}}$   
  coincides with $\{-a^2, -(b\pm c)^2\}$).
   \end{itemize}
 \end{lemma}
 
 \begin{proof} 
 If $g\ne e$ and $x\in \mathcal{O}_g$, then the multiplication operators $L_x$ and $R_x$ are homogeneous maps, which send $\mathcal{O}_{h}$ into $\mathcal{O}_{g+h}$. Thus $\Der(\mathcal{O})_{g}\oplus L_{\mathcal{O}_{g}}\oplus R_{\mathcal{O}_{g}}\subset ( \mathfrak{d}_4)_g$, and, this content is an equality since the sum of $\Der(\mathcal{O})_{g}\oplus L_{\mathcal{O}_{g}}\oplus R_{\mathcal{O}_{g}}$, for $g\ne e$, is the whole algebra $\f{d}_4$ by \eqref{eq_d4}. In particular $\dim( \mathfrak{d}_4)_g=4$ for all $g\ne e$. As $\dim\f{d}_4=28$, it also implies that $ ( \mathfrak{d}_4)_e=0$.
 Thus $[( \mathfrak{d}_4)_g,( \mathfrak{d}_4)_g]\subset ( \mathfrak{d}_4)_e=0$ and the homogeneous components are abelian and Cartan subalgebras.
  (Alternatively,  $[d,L_x]=L_{d(x)}$ and $[d,R_x]=R_{d(x)}$ for all $x\in\mathcal{O}_g$ and $d$ derivation, but, if $d\in \Der(\mathcal{O})_{g}$, then $d(x) \in\mathcal{O}_{2g}\cap \mathcal{O}_0=\mathbb C1 \cap \mathcal{O}_0=0$.)
   The same reasoning works for $\f{b}_3$: 
 $\Der(\mathcal{O})_{g}\oplus \ad_{\mathcal{O}_{g}}\subset ( \mathfrak{b}_3)_g$ if $g\ne e$, a content which is necessarily  an equality (the left-sum fills $\f{b}_3$). So 
 $ ( \mathfrak{b}_3)_e=0$ (alternatively, $ ( \mathfrak{b}_3)_e$ is contained in $ ( \mathfrak{d}_4)_e=0$) and $( \mathfrak{b}_3)_g$ abelian.
 In terms of $\varphi_{x,y}$-operators, $\sum_{h\in G}\varphi_{\mathcal{O}_h,\mathcal{O}_{g+h}}\subset ( \mathfrak{d}_4)_g$ and the equality follows from the fact that $\sum_{g\ne e}\sum_{h\in G}\varphi_{\mathcal{O}_h,\mathcal{O}_{g+h}}=\varphi_{\mathcal{O},\mathcal{O}}=\mathfrak{d}_4$ (we can take the sum with $g\ne e$, since 
 $\varphi_{x,x}=0$ and $\mathcal{O}_g$ is one-dimensional).

 Note that 
 $ z=d\varphi_{e_{0},e_{i}}+a\varphi_{e_{j},e_{i*j}}+b\varphi_{e_{k},e_{i*k}}+c\varphi_{e_{j*k},e_{i*j*k}}\in ( \mathfrak{d}_4)_{g_i}$ satisfies
  $$
 \begin{array}{llll}
 z(e_0)=de_{i},& z(e_{j})=ae_{i*j}, & z(e_{k})=be_{i*k},& z(e_{j*k})=ce_{i*j*k},\\
  z(e_{i})=-de_0,\qquad\quad & z(e_{i*j})=-ae_{j},\qquad\quad & z(e_{i*k})=-be_{k},\qquad\quad & z(e_{i*j*k})=-ce_{j*k}.
 \end{array}
 $$
 From here it is clear that $\mathcal{B}_{ijk}$ is a linearly independent set, and necessarily a basis of $( \mathfrak{d}_4)_{g_i}$.
In particular $e_{0}$ and $e_{i}$ are eigenvectors of $z^2$ related to the eigenvalue $-d^2$,
$e_{j}$ and $e_{i*j}$ are eigenvectors of $z^2$ related to the eigenvalue $-a^2$,
$e_{k}$ and $e_{i*k}$ are eigenvectors of $z^2$ related to the eigenvalue $-b^2$ and 
$e_{j*k}$ and $e_{i*j*k}$ are eigenvectors of $z^2$ related to the eigenvalue $-c^2$.
In order to describe $\ad^2z\vert_{\la_{g_j}}   $, note that 
\begin{equation}\label{eq_auxiliar}
[z ,[ z ,\varphi_{x,y}]]=\varphi_{z ^2(x),y}+2\varphi_{z (x),z (y)}+\varphi_{x,z ^2(y)}.
\end{equation}
Thus $\ad^2z(\varphi_{e_{0},e_{j}})=\varphi_{-d^2e_{0},e_{j}}+2\varphi_{de_{i},ae_{i*j}}+\varphi_{e_{0},-a^2e_{j}}$, which provides the first column
of $\ad^2z\vert_{\la_{g_j}}   $ relative to the basis 
$\mathcal{B}_{jik}= \{\varphi_{e_{0},e_{j}},\varphi_{e_{i},e_{i*j}},\varphi_{e_{k},e_{j*k}},\varphi_{e_{i*k},e_{i*j*k}} \}$. 
Following \eqref{eq_auxiliar}, we have
\begin{equation}\label{eq_calculos}
\begin{array}{l}
\ad^2z(\varphi_{e_{i },e_{i*j }})=\varphi_{-d^2e_{ i},e_{i*k}}+2\varphi_{-de_{ 0},-ae_{j }}+\varphi_{e_{ i},-a^2e_{i*j}}\,,\vspace{1pt}\\
\ad^2z(\varphi_{e_{k },e_{ j*k}})=\varphi_{-b^2e_{ k},e_{j*k}}+2\varphi_{be_{i*k },ce_{i*j*k }}+\varphi_{e_{k },-c^2e_{j*k}}\,,\vspace{1pt}\\
\ad^2z(\varphi_{e_{i*k },e_{ i*j*k}})=\varphi_{-b^2e_{i*k },e_{ i*j*k}}+2\varphi_{-be_{k },-ce_{ j*k}}+\varphi_{e_{ i*k},-c^2e_{ i*j*k}}\,.
\end{array}
\end{equation}
This provides the matrix given in (ii). 
For computing the matrix of $\ad^2z'\vert_{\la_{g_j}}   $ for $\la=\f{b}_3$, relative to the basis $\mathcal{B}'_{jik}$, we can take advantage of the computations in \eqref{eq_calculos}, simply imposing $d=0$. 
Finally, the assertions about the sets of eigenvalues are completely trivial known the matrices in some basis. 
 \end{proof}

 \begin{lemma}\label{co_lasnuestrassongood}
 For any $\la\in\{\mathfrak{d}_4,\mathfrak{b}_3  \}$, the homogeneous components of $\Gamma_\la$ satisfy 
  \begin{itemize}
   \item[\rm(i)] $\la_e=0$;
   \item[\rm(ii)] $[\la_{g_i},\la_{g_j}]=\la_{g_i+g_j}$ if $i,j\in I$;
   \item[\rm(iii)] For $i,j,k$  with $j\ne i$ and $k\notin\ell_{ij}$, there exist $x \in \la_{g_i}, \, y\in \la_{g_j}, \, z\in  \la_{g_k} $  such that $[x,[y,\, z]], \, [y,[z,\, x]]$ are linearly independent.
     \end{itemize}
  
 \end{lemma}
 
 \begin{proof}
 Item (i) has been previously discussed. Now
 compute, for the elements in   $(\f{d}_4)_{g_i}$, $z=\varphi_{e_0,e_i}$ and $w=\varphi_{e_{k},e_{i*k}}$, 
 $$
 \begin{array}{ll}
 \ad z(\varphi_{e_{0},e_{j}})=\varphi_{e_{i},e_{j}},&
  \ad w(\varphi_{e_{k},e_{j*k}})=\varphi_{e_{i*k},e_{j*k}},\\
   \ad z(\varphi_{e_{i},e_{i*j}})=-\varphi_{e_{0},e_{i*j}},\qquad\qquad&
    \ad w(\varphi_{e_{i*k},e_{i*j*k}})=-\varphi_{e_{k},e_{i*j*k}}.
 \end{array}
 $$
 Thus the basis $\mathcal{B}_{i*j,j,k}$ of $ \la_{g_{i*j}}$ is contained in $ [\la_{g_i},\la_{g_j}]$ and hence $\la_{g_i+g_j}\subset [\la_{g_i},\la_{g_j}]$, for $\la=\f{d}_4$.   
 We obtain (ii) for  $\la=\f{b}_3$ in a similar way. 
 
 Item (iii) is clear because we can even find $x \in \Der(\mathcal{O})_{g_i}, \, y\in \Der(\mathcal{O})_{g_j}, \, z\in  \Der(\mathcal{O})_{g_k} $ with that property (\cite[Lemma~2.2]{Thomas1}). Recall that the three gradings are compatible since 
 $\Der(\mathcal{O})_{g}\subset (\mathfrak{b}_3)_g\subset  (\mathfrak{d}_4)_g $ for all $g\in\bb Z_2^3$.
  \end{proof}

  \begin{lemma}\label{le_lineacopiassl2}
  For any $i\ne j$, the subalgebra $\mathcal{H}_{ij}:=\la_{g_i}\oplus \la_{g_j}\oplus \la_{g_i+g_j}$ is semisimple, more precisely  it is the direct sum of $r$ copies of ideals isomorphic to $\mathfrak{sl}(2,\bb C)$, for $r$ the rank of $\la$ ($r=4,3,2$ if $\la=\mathfrak{d}_4,\mathfrak{b}_3 ,\mathfrak{g}_2$, respectively). 
  \end{lemma}
  
  \begin{proof}
 (The $\mathfrak{g}_2$-case has been proved in \cite[Lemma~2.1]{Thomas1}, exhibiting concrete derivations.)
We can assume without loss of generality that $i\prec j$.
  Fix $k\notin\ell_{ij}$. For any $s\in\ell_{ij}=\{i,j,i*j\}$, take $t$ the only index in $\ell_{ij}$ with $s\prec t$, and take the basis 
  $\{x_s,y_s,z_s,w_s\}$ of   $ ( \mathfrak{d}_4)_{g_s}$ 
  given as follows:
 \begin{equation}\label{basedelossl2}
 \begin{array}{ll}
 x_{s}:=\frac12\big(\varphi_{e_{0},e_{s}}+ \varphi_{e_{t},e_{s*t}}\big),&
 y_{s}:=-\frac12\big(\varphi_{e_{0},e_{s}}- \varphi_{e_{t},e_{s*t}}\big),\\
 z_{s}:=\frac12\big(\varphi_{e_{k},e_{s*k}}+ \varphi_{e_{t*k},e_{s*t*k}}\big), \qquad\ &
 w_{s}:=-\frac12\big(\varphi_{e_{k},e_{s*k}}- \varphi_{e_{t*k},e_{s*t*k}}\big); 
 \end{array}
\end{equation}
 that is, the basis obtained from $\mathcal{B}_{stk}$ after the change of basis
 $\frac12\tiny\begin{pmatrix}1&-1&0&0\\1&1&0&0\\0&0&1&-1\\0&0&1&1\end{pmatrix}$.
 Similarly,  consider  ($s\prec t$, both indices in $\ell_{ij}$)
 $$
 u_s:=\varphi_{e_{t},e_{s*t}},
 $$
  and the basis of $ ( \mathfrak{b}_3)_{g_s}$ given by $\{u_s,z_s,w_s\}$, obtained from $\mathcal{B}'_{stk}$ after the change of basis $\frac12  \tiny
  \begin{pmatrix}
  2 & 0 & 0 \\
  0 & 1 & 1 \\
  0 & 1 & -1\\
  \end{pmatrix}$.
  \normalsize
  Now it is straightforward to check that $[x_i,x_j]=x_{i*j}$, 
 $[y_i,y_j]=y_{i*j}$, $[z_i,z_j]=z_{i*j}$, $[w_i,w_j]=w_{i*j}$, $[u_i,u_j]=u_{i*j}$, 
and the same happens  for any cyclic permutation of the three indices in $\ell_{ij}$. 
  We can check, for instance, for $s\prec t$, 
    $$
    \begin{array}{ll}
  [x_s,x_{t}]&=\frac14[\varphi_{e_{0},e_{s}}+ \varphi_{e_{t},e_{s*t}},\varphi_{e_{0},e_{t}}+ 
  \varphi_{e_{s*t},e_{s}}]\\
  &=\frac14(\varphi_{\varphi_{e_{0},e_{s}}(e_{0}),e_{t}}+  \varphi_{e_{s*t},\varphi_{e_{0},e_{s}}(e_{s})}+ \varphi_{e_{0},\varphi_{e_{t},e_{s*t}}(e_{t})}+ \varphi_{\varphi_{e_{t},e_{s*t}}(e_{s*
  t}),e_{s}})\\
  &=\frac14( \varphi_{e_{s},e_{t}}-\varphi_{e_{s*t},e_{0}}+ \varphi_{e_{0},e_{s*t}}-\varphi_{e_{t},e_{s}}) \\
  &=\frac24(\varphi_{e_{0},e_{s*t}}+ \varphi_{e_{s},e_{t}})=x_{s*t}.\\
  \end{array}
  $$
  Also, any bracket   $[a_s,b_{s'}]=0$ if $s,s'\in\ell_{ij}$, both if  $a\ne b\in\{x,y,z,w\}$  and   if $a\ne b\in\{u,z,w\}$.
  This means that $\bb Ca_{i}\oplus\bb Ca_{j}\oplus\bb Ca_{i*j}$
   is one of the required simple ideals for any choice of  $a \in\{x,y,z,w\}$ (respectively, $a\ \in\{u,z,w\}$).
 \end{proof}
 
 \begin{remark}\label{re_fiij}
 This lemma permits  consideration, for any $i\prec j$ and $\la=\mathfrak{d}_4$, of the linear map $\varphi_{ij}\in\mathrm{Aut}(\mathcal{H}_{ij})$ defined by 
 $$
 \varphi_{ij}:a_i\mapsto -a_j,\quad
 a_j\mapsto a_i,\quad
 a_{i*j}\mapsto a_{i*j},\qquad \forall a\in\{x,y,z,w\},
 $$
 with the notations in Eq.~\eqref{basedelossl2}.
 This is an automorphism of the semisimple subalgebra $\mathcal{H}_{ij}\cong 4\mathfrak{sl}(2,\mathbb C)$.
 We will also use the same notation  $\varphi_{ij}$  if  $\la=\mathfrak{b}_3$ (then $a\ \in\{u,z,w\}$).
 We will consider in \S\ref{se_resultadosengrcont} certain linear extensions  to the whole Lie algebra $\la=\mathfrak{d}_4$ (respectively, $\la=\mathfrak{b}_3$) which are not automorphisms of $\la$  but 
  are still   very convenient for us.
 \end{remark}

 Another common feature among these gradings is  their shared Weyl group. Recall that, if $\Gamma$ is a $G$-grading on an algebra $A$, the following groups are remarkable:
 \begin{itemize}
 \item $\Aut(\Gamma) = \{f \in \Aut(A) \colon \mbox{for any } g \in G \mbox{ there exists } g' \mbox{ with  } f(A_g)\subseteq A_{g'}\}$; 

 \item $\Stab(\Gamma)=\{f\in \Aut(A)\colon f(A_g)\subseteq A_g,\ \forall \, g\in A\}$; 
 \item $\weyl(\Gamma)=\Aut(\Gamma)/\Stab(\Gamma)$, the so called \emph{Weyl group} of $\Gamma$. 
 \end{itemize}
 It is  well-known that 
$\weyl(\Gamma_{\mathcal{O}} )\cong \weyl(\Gamma_{\f{g}_2})\cong \Aut(\bb Z_2^3).  $
 For any group automorphism  $\alpha\colon \bb Z_2^3\to\bb Z_2^3$, there exists an algebra automorphism $f_\alpha\colon \mathcal{O}\to \mathcal{O}$ such that 
 $f_\alpha(\mathcal{O}_g)=\mathcal{O}_{\alpha(g)}$. 
 Moreover, the map $d \mapsto  f_\alpha df_\alpha^{-1}$ provides an automorphism of $\f{g}_2=\Der(\mathcal{O})$ which sends 
 $(\f{g}_2)_{\alpha}$ into $(\f{g}_2)_{\alpha(g)}$. 
 The natural extension    $\tilde f_\alpha\colon \f{d}_4\to\f{d}_4$, $\varphi\mapsto f_\alpha\varphi f_\alpha^{-1}$ is a Lie algebra automorphism too, 
 and satisfies $\tilde f((\f{d}_4)_g)=(\f{d}_4)_{\alpha(g)}$.
 Also, its restriction to $\f{b}_3$ satisfies $\tilde f((\f{b}_3)_g)=(\f{b}_3)_{\alpha(g)}$.
  Indeed, 
$\tilde f_\alpha (\varphi_{x,y})=\varphi_{f_\alpha(x),f_\alpha(y)}$.
 Hence $\weyl(\Gamma_{\f{b}_3} )\cong \weyl(\Gamma_{\f{d}_4})\cong \Aut(\bb Z_2^3)  $ too.	
 This is the general linear group $\rm{GL}(3,2)$, which is a renowned simple group of order 168, the only simple group of such order, as proved in \cite{Dummit}, which thus coincides with $\rm{PSL}(2,7)$, the automorphism group of the Fano plane.  
 It is not difficult to find a concrete isomorphism $\rm{GL}(3,2)\cong \rm{PSL}(2,7)$: 
 for any $\alpha\in  \Aut(\bb Z_2^3)$, there is $\sigma\colon I\to I$ with $\alpha({g_i})=g_{\sigma(i)}$ ($\alpha({g_0})=g_{0}$), 
 and clearly $\sigma(i*j)=\sigma(i)*\sigma(j)$. Thus $\sigma$ is a collineation of the Fano plane. 
(Some stories and interpretations about the isomorphism  $\rm{GL}(3,2)\cong \rm{PSL}(2,7)$ can be found in \cite{massobregrupo}.)


     \section{Graded contractions of good $\bb Z_2^3$-gradings}  \label{se_resultadosengrcont}

     On one hand, our aim is to extend the results of \cite{Thomas1}   to $\Gamma_{\f{b}_3}$ and $\Gamma_{\f{d}_4}$, and, on the other hand, 
      we also seek to explore the extent to which these results can be applied to other 
       $\bb Z_2^3$-gradings.
     
     \begin{define}
     A $G=\bb Z_2^3$-grading $\Gamma$ on  a Lie algebra $\la$ will be called a \emph{good} grading  if it satisfies:
     \begin{enumerate}
     \item[(i)] $\la_e=0$;
     \item[(ii)]  $[\la_g,\la_h]=\la_{g+h}$ for all $g\ne h\in G\setminus\{e\}$;
     \item[(iii)]  If  $\langle g, h, k \rangle=G$, then there exist $x\in\la_g$, $y\in\la_h$ and $z\in\la_k$ such that 
     the set $\{[x,[y,z]],[y,[z,x]]\}$ is   linearly independent.
     \end{enumerate}
      If  moreover    $\weyl(\Gamma)\cong\Aut(\bb Z_2^3)  $, then $\Gamma$ will be called a \emph{very good} grading.
         \end{define}

     The name \emph{good}  simply means convenient for our current purposes, although there are very nice $\bb Z_2^3$-gradings on simple Lie algebras not satisfying the above conditions. 
     For instance, the $\bb Z_2^3$-grading on  $\gla(\mathcal{O})$ (simple Lie algebra of type $A_7$) induced by $\Gamma_{\mathcal{O}}$ as in \eqref{eq_gradgl}, does not satisfy the condition (i).
      Instead, our remarkable gradings $\Gamma_{\f{g}_2}$, $\Gamma_{\f{b}_3}$ and $\Gamma_{\f{d}_4}$ are all examples of very good gradings.
      Assume from now on,  throughout the remainder of the manuscript,  that $\Gamma$ is a good grading on a Lie algebra $\la$. 
      
     \begin{define}
We call $\ep\colon G\times G \to \bb{C}$ an \emph{admissible} map if $\ep(g, g) = \ep(e, g) = \ep(g, e) = 0$, for all $g \in G$.
\end{define}

 Admissible maps are important for our classifications due to the following result.

\begin{lemma}
For any graded contraction $\ep$   of   $\Gamma$,  
 there exists an admissible graded contraction $\ep'$ of $\Gamma$ which is equivalent to $\ep$.
\end{lemma}

The proof is  identical to \cite[Lemma~3.2]{Thomas1}, due to (i) (do not forget we are assuming that all the gradings are good). It is enough to define $\ep'(g,h)=\ep(g,h)$ if $g,h,g+h$ are all of them different from $e\in G$, and $\ep'(g,h)=0$ otherwise. 
 
Thus, classifying graded contractions up to equivalence is the same problem as that of classifying admissible graded contractions up to equivalence.
The advantage of dealing with admissible graded contractions is that they can be easily translated to a combinatorial setting, we mean, where the elements in $\la$ do not play any role.

\begin{prop}\label{ref_existeadmisible}
An admissible map $\ep\colon  G \times G \to \bb{C}$ is a graded contraction of   $\Gamma$ if and only if the following conditions hold for all $g, h, k \in G$:
\begin{enumerate}
\item[\rm (b1)] $\ep(g, h) = \ep(h, g)$,  
\item[\rm (b2)] $\ep(g, h, k) = \ep(k, g, h)$, provided that $G = \langle g, h, k \rangle$. 
\end{enumerate}

\end{prop}

The proof mimics the proof of \cite[Proposition~3.4]{Thomas1}, because the only relevant facts   used there were precisely the conditions (i), (ii) and (iii) satisfied by $\Gamma_{\mathfrak g_2}$.

\begin{remark}  
 For any admissible graded contraction $\ep\colon G\times G\to\bb C$, write $\ep_{ij}:=\ep(g_i,g_j)$ and $\ep_{ijk}:=\ep(g_i,g_j,g_k)$.
 Note that  $\ep_{ij} $ does not depend on the order of the indices, 
 so we can attach  to the graded contraction
  the map $\ep\colon X\to \bb C$  defined by 
$\ep({\{i,j\}}):=\ep_{ij}$, for
 $X:=\{\{i,j\}:1\le i<j \le7\}$ the set of the 21 edges of the Fano plane. 
Taking into account the above proposition,  
 there is a one-to-one correspondence between admissible graded contractions of $\Gamma$ 
 and maps $\ep\colon X\to \bb C$ such that the value $\ep_{ijk}:= \ep_{i,j*k}\ep_{jk}$ does not depend on the order of the indices whenever 
 $\langle g_i,g_j,g_k \rangle= G$ (equivalently, when the expression $\ep_{ijk}$ is well defined, that is, $j\ne k$, $i\notin\ell_{jk}$).
 From now on, we refer  interchangeably to the map $\ep\colon X\to \bb C$ and to the unique related admissible graded contraction.
\end{remark}

\begin{define}
For any admissible graded contraction $\ep$, 
define its \emph{support} by $\supp\ep:=\{\{i,j\}:\ep_{ij}\ne0\}\subset X.$ 
\end{define}

The support is an invariant under strong  equivalence.
\begin{lemma}\label{le_soportesiguales}
If $\ep\approx \ep'$ are admissible graded contractions of $\Gamma$, then $\supp\ep=\supp{\ep'}$.
\end{lemma}
\begin{proof}
If $\varphi\colon \la^{\ep}\to\la^{\ep'}$ is a graded isomorphism, $x\in \la_{g_i}$ and $y\in \la_{g_j}$, then
$\ep_{ij}\varphi([x,y])=\ep'_{ij}[\varphi(x),\varphi(y)]$, so that $\ep_{ij}=0$ if and only if $\ep'_{ij}=0$, due to (ii).
\end{proof}

The supports are characterized by a certain absorbing property.
 If $k\notin\ell_{ij}$, we consider 
\[
P_{\{i,j,k\}} := \{\{i, j\}, \{j,k\},  \{k,i\}, \{i, j\ast k\}, \{j, k\ast i\}, \{k, i\ast j\}\}.
\]
\begin{define} \label{def_nice}
 A subset $T\subseteq X$ is called \emph{nice} if $\{i, j\},   \{i\ast j, k\} \in T$ implies $P_{\{i,j,k\}} \subseteq T$, for all $i, j, k \in I$ with $i\ne j$, $k\notin\ell_{ij}$.
 As proved in \cite[Proposition~3.10]{Thomas1}, the support of any admissible graded contraction is a nice set. 
\end{define}

Some remarkable examples of nice sets are, for any $\ell\in\mathbf{L}$ and for any $i\in I$,
\begin{itemize}
\item 
$X_{\ell}:=\{\gamma\in X:\gamma\subset \ell\}$;
\item 
$X_{\ell^C}:=\{\gamma\in X:\gamma\cap\ell=\emptyset\}$; 
\item $X_{(i)}=\{\gamma\in X:i\in \gamma\}$;
\item $X^{(i)}=\{\{j,k\}\in X:j*k=i\}$.
\end{itemize}
Also, if $k\notin\ell_{ij}$, the following set of cardinal 10 is nice,
$$T_{ijk}:= P_{\{i,j,k\}}\cup \{\{ i,i*j \},\{ i,i*k \},\{ i*j,i*k \},\{ i,i*j*k \} \}.$$  
 Moreover, the list of all possible supports follows.
 
\begin{theorem}\cite[Propositions~3.23 and 3.25]{Thomas1} 
The complete list of nice sets is: $X$,  $X\setminus  X_{\ell^C}$, $P_{\{i,j,k\}}$, $T_{ijk}$, $X^{(i)}$ and any subset of either $X_{\ell}$, or $X_{\ell^C}$, or $X_{(i)}$,   for some $\ell\in\mathbf{L}$, $i\ne j\in I$, $k\notin\ell_{ij}$.
\end{theorem}

To be precise \cite[Remark 3.29]{Thomas1}, there are 779 nice sets,  which thus are the possible supports of the admissible graded contractions of any good $\bb Z_2^3$-grading $\Gamma$, although the explicit list is not relevant for our purposes.
Conversely, each nice set provides an admissible graded contraction of $\Gamma$, as in \cite[Proposition~3.11]{Thomas1}. 

\begin{define}  
If $T\subseteq X$ is a nice set, define $\ep^T\colon G\times G\to\bb C$ by $\ep^T_{ij}=1$ if $\{i,j\}\in T$ and $\ep^T_{ij}=0$ otherwise. 
Observe that $\ep^T$ is an admissible graded contraction.
\end{define}

Our aim is to prove that the support  of a graded contraction determines its strong-equivalence class in most of the cases. 
 As a preview, we will say that whenever the nice set $T$ 
 is not any of the following (for some $j\ne i$)
 \begin{equation}\label{eq_malos}
 X_{(i)}, \qquad X_{(i)}\setminus\{\{i,j\} \}, \qquad X_{(i)}\setminus\{\{i,j\},\{i,i*j\} \},
\end{equation}
 then any admissible graded contraction with support $T$ is necessarily strongly equivalent to $\ep^T$. 
The main tool to prove it is the notion of equivalence by normalization, a concept inspired in \cite{checos}  which is  combinatorial in nature and avoids having to delve into the bracket of the Lie algebra in most cases.

\begin{define}
For any maps $\alpha\colon  G \to \mathbb{C}^\times$ and   $\ep\colon G\times G\to\bb C$,
define $\ep^\alpha\colon G\times G\to\bb C$ by 
$
\ep^\alpha(g, h) := \ep(g, h) \dfrac{\alpha(g)\alpha(h)}{\alpha(g + h)}
$.
If $\ep$ is a
 graded contraction   of $\Gamma$,   then  $\ep^\alpha$ so is.
In such case we write $\ep\sim_n\ep^\alpha $, and say that $\ep$ and $\ep^\alpha $ are \emph{equivalent by normalization.}
Note that $\ep^\alpha$ is admissible if and only if $\ep$  is also, and in that case $\supp{\ep}=\supp{\ep^\alpha} $. Besides, 
$\ep \sim_n \ep'$  implies  $\ep \approx \ep'$, because the map $f\colon \la^\ep\to\la^{\ep^\alpha}$ given by  $f\vert_{ \la_{g_i}}=\alpha(g_i)\id_{ \la_{g_i}}$   is a graded-isomorphism.
\end{define}

The following theorem, adapted from \cite[Theorem~4.1]{Thomas1} to our  current setting, describes the equivalence classes up to normalization of admissible graded contractions. 

\begin{theorem} \label{th_normaliz}
  
 If $i,j,k\in I$, $j\ne i$, $k\notin\ell_{ij}$, $\lambda, \lambda'\in\bb C^\times$, denote by 
 $\eta^\lambda\colon X\to \bb C ,$ $\mu^\lambda\colon X\to \bb C ,$ and $\beta^{\lambda,\lambda'} \colon X\to \bb C ,$ the 
 admissible graded contractions 
 defined by 
 $$
 \begin{array}{l}
 \eta^\lambda_{ij}=\lambda,\quad
 \eta^\lambda_{i,i*j}=\eta^\lambda_{i,k}=\eta^\lambda_{i,i*k}=1,\vspace{3pt}\\
 \mu^\lambda_{ij}= \mu^\lambda_{i,i*j}=\lambda, \quad
 \mu^\lambda_{ik}= \mu^\lambda_{i,i*k}= \mu^\lambda_{i,j*k}=1,\vspace{3pt}\\
 \beta^{\lambda,\lambda'}_{ij}=\beta^{\lambda,\lambda'}_{i,i*j}=\lambda,\quad
 \beta^{\lambda,\lambda'}_{ik}=\beta^{\lambda,\lambda'}_{i,i*k}=\lambda',\quad
 \beta^{\lambda,\lambda'}_{i,j*k}=\beta^{\lambda,\lambda'}_{i,i*j*k}=1,
 \end{array}
 $$
 and 0 otherwise.
 
 Take $\ep$ an admissible graded contraction of a (good grading) $\Gamma$ and let $T=\supp{\ep}$.
 
\begin{itemize}
 \setlength\itemsep{0.5em}
\item  If $T$ is not one of the nice sets in \eqref{eq_malos}, then $\ep\sim_n\ep^T$;

\item If $T=X_{(i)}$, then there are   $\lambda, \lambda'\in\bb C^\times$ such that $\ep\sim_n\beta^{\lambda,\lambda'}$;
\item If $T=X_{(i)}\setminus\{\{i,i*j*k\}\}$, then there is   $\lambda\in\bb C^\times$ such that $\ep\sim_n\mu^{\lambda}$;
\item If $T=X_{(i)}\setminus\{\{i,j*k\},\{i,i*j*k\}\}$, then there is a unique $\lambda\in\bb C^\times$ such that $\ep\sim_n\eta^{\lambda}$.
\end{itemize}
Moreover, $\beta^{\lambda_1,\lambda_2}\sim_n\beta^{\lambda'_1,\lambda'_2}$ if and only if $\lambda_1=\pm\lambda'_1$ and $\lambda_2=\pm\lambda'_2$; and
  $\mu^{\lambda}\sim_n\mu^{\lambda'}$ if and only if $\lambda= \pm\lambda'$.

\end{theorem}

As our purpose is to classify graded contractions up to equivalence (equivalently, admissible graded contractions up to equivalence), a  key first
step is to 
classify admissible graded contractions up to strong equivalence. (After that, the  Weyl group will help us to relate the two classifications.)
Similar to the $\f{g}_2$-case, strong equivalence and equivalence by normalization will also coincide
 for admissible graded contractions of both $\Gamma_{\f{d}_4}$ and $\Gamma_{\f{b}_3}$, but adapting that proof is not straightforward. 
Observe that  \cite[Conjecture 2.15]{ref33} wonders whether it is always true that strong equivalence and equivalence by normalization  coincide; a counterexample is exhibited in \cite[Example 2.12]{Thomas1}.

 \begin{theorem}\label{teo_fuerte}
 Let $\ep,\ep'$ be two admissible graded contractions of  either $\Gamma_{\f{d}_4}$ or $\Gamma_{\f{b}_3}$. Then  $\ep\approx\ep'$ if and only if 
  $\ep\sim_n\ep'$. 
 \end{theorem}

\begin{proof} 
 We begin by considering $\psi\colon \la^{\ep}\to \la^{\ep'}$ a graded-isomorphism between admissible graded contractions of a good grading on $\la$, and 
a pair $i,j$ with $\{i,j\},\{i,i*j\}$ belonging to the support $ T=\supp\ep$ (necessarily equal to $\supp{\ep'}$, by Lemma~\ref{le_soportesiguales}).
 In such a case, for   $z \in \la_{g_i}\equiv \la_i$ and  $w\in \la_j$ we have 
$\psi([z,[z,w]^{\ep } ]^{\ep } )=[\psi(z),[\psi(z),\psi(w)]^{\ep' } ]^{\ep' } $,
so that
$\ep_{ij}\ep_{i\,i\ast j}\psi([z,[z,w]])=\ep'_{ij}\ep'_{i\,i\ast j}[\psi(z),[\psi(z),\psi(w)]]$. 
 This implies 
$\ad^2\psi(z)\vert_{\la_j}= \frac{\ep_{ij}\ep_{i\,i\ast j} }{\ep'_{ij}\ep'_{i\,i\ast j}}  \psi \circ\ad^2z\circ  \psi^{-1}\vert_{\la_j},$
and in particular,  for any $z\in\la_i$,
\begin{equation}\label{eq_relacionespectros}
\Spec(\ad^2\psi(z)\vert_{\la_j})= \frac{\ep_{ij}\ep_{i\,i\ast j} }{\ep'_{ij}\ep'_{i\,i\ast j}} \Spec(\ad^2z\vert_{\la_j}).
\end{equation}

Now, take $\la=\f{b}_3$.
By reductio  ad absurdum, assume the existence of two admissible graded contractions $\ep$ and $\ep'$  which are strongly equivalent but not equivalent by normalization. Denote by  $T=\supp\ep=\supp{\ep'}$.
By  Theorem~\ref{th_normaliz}, $T$ has to be one of the nice sets in \eqref{eq_malos} (for any  other nice set $T$, there is a unique class 
up to normalization of admissible graded contractions with support $T$). Let us look at each of the three possible supports separately.\smallskip

\boxed{T=\{\{i,j\},\{i,i*j\},\{i,k\},\{i,i*k\}\}} for some $j\ne i$, $k\notin\ell_{ij}$.\smallskip

There are $\lambda,\lambda'\in\mathbb C^\times$ such that $\ep\sim_n\eta^{\lambda}$ and  $\ep'\sim_n\eta^{\lambda'}$ by Theorem~\ref{th_normaliz}. As equivalence by normalization implies strong equivalence, in particular
$\eta^{\lambda}$ and $\eta^{\lambda'}$ are strongly equivalent.  
In other words we can assume that $\ep=\eta^{\lambda}$ and $\ep'=\eta^{\lambda'}$.
Take as above a graded isomorphism $\psi\colon \la^{\ep}\to \la^{\ep'}$.
For any $z\in  \la_i $, by \eqref{eq_relacionespectros},
\begin{equation}\label{harta}
\Spec(\ad^2\psi(z)\vert_{\la_j})=\frac{\lambda}{\lambda'}\Spec(\ad^2z\vert_{\la_j}),\qquad
\Spec(\ad^2\psi(z)\vert_{\la_k})= \Spec(\ad^2z\vert_{\la_k}).
\end{equation}
Take, for instance, $z=\varphi_{e_j,e_{i*j}}\in  \la_i $. Lemma~\ref{le_corchete} (ii) says that 
$\Spec(\ad^2z\vert_{\la_j})=\{-1,0,0\}$  and that
$\Spec(\ad^2z\vert_{\la_k})=\{0,-1,-1\}$. 
(This notation takes into account the  multiplicity, but not the order.)
By \eqref{harta},
$\Spec(\ad^2\psi(z)\vert_{\la_j})=\{-\frac{\lambda}{\lambda'},0,0\}$  and
$\Spec(\ad^2\psi(z)\vert_{\la_k})=\{0,-1,-1\}$.
As there are $a,b,c\in\bb C$ such that 
$\psi(z)=a\varphi_{e_{j},e_{i*j}}+b\varphi_{e_{k},e_{i*k}}+c\varphi_{e_{j*k},e_{i*j*k}}$, then, by  again applying
Lemma~\ref{le_corchete} (ii), we get
 $\{-a^2,-(b\pm c)^2\}=\{-\frac{\lambda}{\lambda'},0,0\}$  and
$\{-b^2,-(a\pm c)^2  \}=\{0,-1,-1\}$. As the first two sets are equal, one of the following possibilities  must occur:
\begin{itemize}
\item $a^2=\frac{\lambda}{\lambda'}$, $b+c=b-c=0$. Then $b=c=0$ and   $\{0,-1,-1\}=\{-b^2,-(a\pm c)^2  \}=\{0,-a^2,-a^2\}$. So $\frac{\lambda}{\lambda'}=a^2=1$ and $\lambda=\lambda'$, getting a contradiction (this gives the same class up to normalization).
\item $a=b+c=0$ and $(b-c)^2=\frac{\lambda}{\lambda'}$. Then $\{0,-1,-1\}=\{-b^2,-(0\pm c)^2  \}=\{-b^2,-b^2,-b^2\}$, which is evidently a contradiction.
\item $a=b-c=0$ and $(b+c)^2=\frac{\lambda}{\lambda'}$. Here the contradiction is achieved just as  in the bullet above.
\end{itemize}

\boxed{T=\{\{i,j\},\{i,i*j\},\{i,k\},\{i,i*k\}, \{i,j*k\}\}} for some $j\ne i$, $k\notin\ell_{ij}$.\smallskip

In this case we can assume that $\ep=\mu^{\lambda}$ and $\ep'=\mu^{\lambda'}$ for $\lambda,\lambda'\in\mathbb C^\times$.
Thus $ \frac{\ep_{ij}\ep_{i\,i\ast j} }{\ep'_{ij}\ep'_{i\,i\ast j}}=\frac{\lambda^2}{(\lambda')^2}$
while $ \frac{\ep_{ik}\ep_{i\,i\ast k} }{\ep'_{ik}\ep'_{i\,i\ast k}}=1$.
By applying Lemma~\ref{le_corchete} (ii) to $z=\varphi_{e_j,e_{i*j}}$, we get as above
$\Spec(\ad^2z \vert_{\la_j})=\{-1,0,0\}$ and 
$\Spec(\ad^2z \vert_{\la_k})=\{0,-1,-1\}$. This time  Eq.~\eqref{eq_relacionespectros} gives 
$\Spec(\ad^2\psi(z )\vert_{\la_j})=\{-\left(\frac{\lambda}{\lambda'}\right)^2,0,0\}$  and
$\Spec(\ad^2\psi(z )\vert_{\la_k})=\{0,-1,-1\}$.
The same analysis as above forces $a^2=\left(\frac{\lambda}{\lambda'}\right)^2$, $b=c=0$,
for $a,b,c\in\bb C$ the coefficients of $\psi(z)$ with respect to the basis $\mathcal B'_{ijk}$. Thus $\left(\frac{\lambda}{\lambda'}\right)^2=1$,
and $\lambda=\pm\lambda'$, so that $\mu^{\lambda}\sim_n\mu^{\lambda'}$, a contradiction.\smallskip

\boxed{T=X_{(i)}}  
\smallskip

Assume now that $\ep=\beta^{\lambda_1,\lambda_2}$ and $\ep'=\beta^{\lambda'_1,\lambda'_2}$ are  strongly equivalent admissible  graded contractions,
for the maps $\beta$'s defined in Theorem~\ref{th_normaliz}, $\lambda_i,\lambda'_i\in\mathbb C^\times$  for both of $i=1,2$.
For $z=\varphi_{e_j,e_{i*j}}\in\la_i$,  if $\psi(z)=  a\varphi_{e_{j},e_{i*j}}+b\varphi_{e_{k},e_{i*k}}+c\varphi_{e_{j*k},e_{i*j*k}}$, then Eq.~\eqref{eq_relacionespectros} tells
\begin{itemize}
\item[(i)] $
\{-a,-(b\pm c)^2\}=\Spec(\ad^2\psi(z)\vert_{\la_j})= \frac{\lambda^2_1}{(\lambda'_1)^2}  \Spec(\ad^2z\vert_{\la_j})=\{
\frac{-\lambda^2_1}{(\lambda'_1)^2},0,0\}$;
\item[(ii)] $
\{-b,-(a\pm c)^2\}=\Spec(\ad^2\psi(z)\vert_{\la_k})= \frac{\lambda^2_2}{(\lambda'_2)^2} \Spec(\ad^2z\vert_{\la_{k}})=\{0,
\frac{-\lambda^2_2}{(\lambda'_2)^2},\frac{-\lambda^2_2}{(\lambda'_2)^2}\}$;
\item[(iii)] $
\{-c,-(a\pm b)^2\}=\Spec(\ad^2\psi(z)\vert_{\la_{j*k}})=   \Spec(\ad^2z\vert_{\la_{j*k}})=\{0,-1,-1\}$.
\end{itemize}
If $a=0$ then either $b+ c=0$ or $b-c=0$, and the sets in (ii) and (iii) would have a unique value with multiplicity 3, contradiction. This leads to $b=c=0$, so that 
$a^2=\frac{\lambda^2_1}{(\lambda'_1)^2} $ (by (i)). By (ii), $a^2=\frac{\lambda^2_2}{(\lambda'_2)^2} $, and by (iii), $a^2=1$.
We have proved   $\left(\frac{\lambda_1}{\lambda'_1}\right)^2=1=\left(\frac{\lambda_2}{\lambda'_2}\right)^2$, so $\ep$ and $\ep'$ are equivalent by normalization, again due to Theorem~\ref{th_normaliz}. This finishes the proof for $\la=\f{b}_3$. \smallskip

Now, take $\la=\f{d}_4$ and again we are going to prove that there can not exist 
  two admissible graded contractions $\ep$ and $\ep'$  which are strongly equivalent but not equivalent by normalization. If there were, the support  $T=\supp\ep=\supp{\ep'}$ would be one of the nice sets in \eqref{eq_malos}.
Let us check the case $T=\{\{i,j\},\{i,i*j\},\{i,k\},\{i,i*k\}\}$ for some $j\ne i$, $k\notin\ell_{ij}$, since the two other possible supports work similarly. 
There are $\lambda,\lambda'\in\mathbb C^\times$ such that $\ep\sim_n\eta^{\lambda}$ and  $\ep'\sim_n\eta^{\lambda'}$ by Theorem~\ref{th_normaliz},
so  $\eta^{\lambda}$ and $\eta^{\lambda'}$ are strongly equivalent and there exists    a graded isomorphism
$\psi\colon \la^{\eta^{\lambda}}\to \la^{\eta^{\lambda'}}$.  Call $\psi(z)= d \varphi_{1,e_{i}}+ a\varphi_{e_{j},e_{i*j}}+b\varphi_{e_{k},e_{i*k}}+c\varphi_{e_{j*k},e_{i*j*k}}$ the image of $z=\varphi_{e_j,e_{i*j}}\in\la_i$. According to Eq.~\eqref{eq_relacionespectros} and to Lemma~\ref{le_corchete},
\begin{itemize}
\item[(i)] $
\{-(a\pm d)^2,-(b\pm c)^2\}=\Spec(\ad^2\psi(z)\vert_{\la_j})= \frac{\lambda}{\lambda'}  \Spec(\ad^2z\vert_{\la_j})=\{
\frac{-\lambda}{\lambda'},\frac{-\lambda}{\lambda'},0,0\}$;
\item[(ii)] $
\{-(b\pm d)^2,-(a\pm c)^2\}=\Spec(\ad^2\psi(z)\vert_{\la_k})=  \Spec(\ad^2z\vert_{\la_{k}})=\{0,
0,-1,-1\}$.
\end{itemize}
Looking at (i), we have three possible situations. First, if $b=c=0$, then either $a+d=a-d$ or $a+d=d-a$.
 In the first case, $d=0$; and in the second one, $a=0$. 
  In the first case, $a^2=\frac{\lambda}{\lambda'} $, but (ii) gives
$\{-a^2,-a^2,0,0\}=\{0,0,-1,-1\}$, so $\frac{\lambda}{\lambda'} =a^2=1$ and ${\lambda}={\lambda'} $. In the second case, 
$d^2=\frac{\lambda}{\lambda'} $, but (ii) gives
$\{ -d^2,-d^2,0,0\}=\{0,0,-1,-1\}$, getting again ${\lambda}={\lambda'} $. The second situation is $a=d=0$ and either $b+c=b-c$ or $b+c=c-b$. This   can be discussed in a completely similar way to the first situation. The third possibility is that there are $s,t\in\{\pm1\}$ with $a+sd=0=b+tc$
and $(a-sd)^2=(b-tc)^2=\frac{\lambda}{\lambda'} $. Thus $(2a)^2=(2b)^2=\frac{\lambda}{\lambda'} $ and $a=\pm b$, so that any pair of scalars in $\{a,b,c,d\}$ are either equal or opposite. This gives $\{b\pm d\}=\{2b,0\}$ (not ordered sets) and $\{a\pm c\}=\{2a,0\}$, so that (ii) yields $\frac{\lambda}{\lambda'} =1$. 
The rest of the proof consists of analyzing the other possible supports, using the same arguments.
\end{proof} 

\begin{cor}
The Lie algebras obtained from $\la\in\{\f{d}_4,\f{b}_3\}$ by graded contraction of $\Gamma_\la$ are graded-isomorphic to either
  \begin{itemize}
  \item $\la^{\ep^T}$ for $T$ a nice set,
  
  or
  \item $\la^{\eta^\lambda},\la^{\mu^\lambda}$ and $\la^{\beta^{\lambda,\lambda'} }$ for $\lambda,\lambda'\in\bb C^\times$ and
   the graded contractions $\eta^\lambda,\mu^\lambda,\beta^{\lambda,\lambda'} $ described in Theorem~\ref{th_normaliz}.
  
  \end{itemize}

\end{cor}
 \noindent  (Note that these graded contractions $\eta^\lambda,\mu^\lambda,\beta^{\lambda,\lambda'} $ do not depend exclusively on 
$\lambda, \lambda'\in\bb C^\times$, but on the choice of the indices $i,j,k$, but making the indices explicit would have complicated the notation too much.)

Most remarkably,  
in this corollary  we have not been worried about excluding possibly repeated cases, with the aim of providing a less specific summary.  
 
Now we are prepared to deal with 
the most important question for us, namely,  how many not isomorphic Lie algebras can be obtained by graded contractions of our grading $\Gamma_\la$ (and their properties).  That is, we want to compute the classes of graded contractions by $\sim$; or equivalently the classes of admissible graded contractions by $\sim$.
At this stage,  the Weyl group plays an important role. Recall that the Weyl group   $\mathcal W(\Gamma_\la)=\Aut(\bb Z_2^3)\cong\mathrm{GL}(3,2)$ coincides with the group of collineations of the Fano plane:

\begin{define}
A map $\sigma\colon I\to I$ is called a \emph{collineation}  if  $\sigma(i*j)=\sigma(i)*\sigma(j)$ for all $i\ne j$. 
That is, $\sigma$ sends lines of the Fano plane to lines. 
 \end{define}

The collineation group also acts on $X$ and on the subsets of $X$: 
 \begin{define}
 Two subsets $T,T'$ of $X$ are called \emph{collinear} if there is a collineation $\sigma$ such that $\sigma(T)=T'$. In such case, we will write $T\sim_cT'$.
\end{define}

Note that if $T$ is nice, $\sigma(T)$ also is, so that $\sim_c$ is an equivalence relation in the set of nice sets too.
 According to \cite[Theorem 3.27]{Thomas1},
the  779 nice sets   are grouped in just 24 equivalence classes.
As we mentioned earlier,  the Weyl group begins to play a role, since  it permits us to prove that $\ep^T\sim \ep^{\sigma(T)}$.  
This result is slightly stronger.

\begin{lemma}
If $\Gamma$ is a very good grading, 
the group of collineations acts on the set of admissible graded contractions of $\Gamma$ without changing the equivalence class.
More precisely, if $\ep$ is an admissible graded contraction  then $\sigma\cdot\ep$ is also, for $(\sigma\cdot\ep)_{ij}:=\ep_{\sigma(i)\sigma(j)}$,
and $\sigma\cdot\ep\sim\ep$.
\end{lemma}

\begin{proof}
The map $G\to G$, $g_i\mapsto g_{\sigma(i)}$, $g_0\mapsto g_0$, is a group automorphism. 
As the Weyl group of the grading coincides with ${\rm Aut}(G)$, there is a Lie algebra automorphism $\varphi\colon\la\to\la$ such that $\varphi(\la_{g_i})=\la_{g_{\sigma(i)}}$. Now it is trivial to check that
the same map is a Lie algebra isomorphism when considered as $\varphi\colon\la^{\sigma\cdot\ep}\to\la^\ep$. 
\end{proof}

\begin{cor}\label{cor_pasandoporcol}
If $\Gamma$ is a very good grading, and $T$ and $T'$ are collinear nice sets, then $\ep^T\sim\ep^{T'}$.
\end{cor}

\begin{proof}
Take $\sigma$ a collineation such that $\sigma(T)=T'$ and check that $\sigma\cdot\ep^{T'}=\ep^{T}$.  Then, we can apply the above lemma.
\end{proof}

According to this easy but key result, one can describe the algebras related to the 24 representatives of the equivalence classes of the supports up to collineations. Furthermore, this description can be provided independently of how many non isomorphic contracted Lie algebras have the same support, since many of the properties of the Lie  algebra depend only on the support. A last important observation is 
 that the Lie algebras obtained by graded contractions of $\Gamma_{\f{d}_4}$ and $\Gamma_{\f{d}_3}$ have similar properties to those ones obtained by graded contractions of $\Gamma_{\f{g}_2}$. Then  the next important result follows trivially.

 \begin{theorem}\label{teo_lasalgebras}
Let $\ep\colon G\times G\to \mathbb C$ be an admissible graded contraction of $\Gamma_{\la}$, for $\la\in\{\f{b}_3,\f{d}_4\}$ and take $T=S^{\ep}$. 
Denote  by $r$ the rank of $\la$ (equal to $3$ and $4$ respectively). Then $T$ is collinear to one and only one nice set in the following list:
\smallskip   

$\bullet$ $T_1 = \emptyset$, in which case $\lae$ is abelian. 

$\bullet$ $T_{2}=\{\{1, 2\}\}$, in which case   $\lae$ is 2-step nilpotent, $\dim(\lae)' = r$ and $\dim\f{z}(\la^\ep) = 5r$.

$\bullet$ $T_{3}=\{\{1, 2\}, \{1, 3\}\}$, in which case   $\lae$ is 2-step nilpotent, $\dim(\lae)' = 2r$ and $\dim\f{z}(\la^\ep) = 4r$.

$\bullet$  $T_{4}=\{\{1, 2\}, \{1, 5\}\}$, in which case  $\lae$ is 2-step solvable but not nilpotent, $\dim(\lae)' = 2r$ and  $\dim\f{z}(\la^\ep) = 4r$.

$\bullet$ $T_{5}=\{\{1, 2\}, \{6, 7\}\}$, in which case  $\lae$ is 2-step nilpotent, $\dim(\lae)'= r$ and $\dim\f{z}(\la^\ep) = 3r$.

$\bullet$  $T_{6}=X_{\ell_{12}} =  \{\{1, 2\}, \{1, 5\}, \{2, 5\}\}$, 
in which case  $\lae$ is reductive, neither nilpotent nor solvable, $\lae = \f{z}(\lae) \oplus (\lae)'$
with $\dim\f{z}(\la^\ep) = 4r$. 

$\bullet$ $T_{7}=X^{(1)} = \{\{2, 5\}, \{3, 6\}, \{4, 7\}\}$, 
in which case   $\lae$ is 2-step nilpotent with $\f{z}(\la^\ep) = (\lae)'$ of dimension $r$.

$\bullet$  $T_{8}=\{\{1, 2\}, \{1, 3\}, \{1, 4\}\}$, in which case  $\lae$ is 2-step nilpotent and $\f{z}(\la^\ep) = (\lae)'$  has dimension
$ 3r$.

$\bullet$  $T_{9}=\{\{1, 2\}, \{1, 3\}, \{1, 5\}\}$, in which case   $\lae$ is 2-step solvable  not nilpotent, and $\dim(\lae)' =  \dim\f{z}(\la^\ep) = 3r$.

$\bullet$ $T_{10}=\{\{1, 2\}, \{1, 3\}, \{1, 7\}\}$, in which case    $\lae$ is 2-step nilpotent and $\dim (\la^\ep)' = \dim\f{z}(\la^\ep) = 3r$. 

$\bullet$ $T_{11}=\{\{1, 2\}, \{1, 6\}, \{2, 6\}\}$, in which case   $\lae$ is 2-step nilpotent, $\dim(\lae)' = 3r$  and $\dim\f{z}(\la^\ep) = 4r$.

$\bullet$  $T_{12}=\{\{1, 2\}, \{1, 6\}, \{6, 7\}\}$, in which case   $\lae$ is 2-step nilpotent, $\dim(\lae)' = 2r$ and $\dim\f{z}(\la^\ep) = 3r$.

$\bullet$ 
$T_{13}=\{\{1, 2\}, \{1, 3\}, \{1, 4\}, \{1, 5\}\}$ or $T_{14}=\{\{1, 2\}, \{1, 3\}, \{1, 5\}, \{1, 6\}\}$. In both cases,     $\lae$ is 2-step solvable  not nilpotent, $\dim\f{z}(\la^\ep) = 2r$ and $\dim(\lae)' = 4r$.

$\bullet$  $T_{15}=\{\{1, 2\}, \{1, 6\}, \{1, 7\}, \{2, 6\} \}$, in which case    $\lae$ is 2-step nilpotent  and $\dim (\lae)' = \dim\f{z}(\la^\ep) = 3r$.
 
$\bullet$ $T_{16}= \{\{1, 2\}, \{1, 6\}, \{2, 7\}, \{6, 7\}\}$, in which case   $\lae$ is 2-step nilpotent,  $\dim(\lae)' = 2r$ and $\dim\f{z}(\la^\ep) = 3r$.
 
$\bullet$  $T_{17}=\{\{1, 2\}, \{1, 3\}, \{1, 4\}, \{1, 5\}, \{1, 6\}\}$, in which case   $\lae$ is 2-step solvable not nilpotent,  $\dim(\lae)' = 5r$ and $\dim\f{z}(\la^\ep)=r$.

$\bullet$ $T_{18}=\{\{1, 2\}, \{1, 6\}, \{1, 7\}, \{2, 6\}, \{2, 7\}\}$ or $T_{19}=X_{\ell^C_{12}}$. 
In any of these cases,  $\lae$ is 2-step nilpotent and $\dim (\la^\ep)' = \dim\f{z}(\la^\ep) = 3r$.

$\bullet$  $T_{20}=X_{(1)} =\{\{1, 2\}, \{1, 3\}, \{1, 4\}, \{1, 5\}, \{1, 6\},\{1, 7\}\}$, 
in which case    $\lae$ is 2-step solvable not nilpotent, $\dim(\lae)' = 6r$ and 
$\f{z}(\la^\ep) = 0$.

$\bullet$  $T_{21}=P_{\{1, 2, 3\}} =\{\{1, 2\}, \{1, 3\}, \{1, 7\}, \{2,3\}, \{2, 6\},\{3,5\}\}$,
 in which case   $\lae$ is 2-step solvable, 3-step nilpotent, $\dim(\lae)' = 4r$ 
and $\dim\f{z}(\la^\ep) = r$.

$\bullet$  $T_{22}=T_{\{1,2,3\}}$, in which case   $\lae$ is 3-step solvable not nilpotent, $\dim(\lae)' = 6r$ and $\f{z}(\la^\ep) = 0$.


$\bullet$  $T_{23}=X\setminus X_{\ell^C_{12}}$, in which case  $\lae$ is neither nilpotent nor solvable nor reductive, $\dim\f{r}(\lae) = 4r$ and the 
Levi subalgebra of $\lae$ is  
direct sum of $r$ copies of $\f{sl}(2,\bb C)$.

$\bullet$  $T_{24} = X$, in which case  $\lae$ is simple.

All these algebras $\lae$ are $\bb Z_2^3$-graded, and every homogeneous element is semisimple.
\end{theorem}

Last property allows us to  easily choose bases of $\la^\ep$ formed by semisimple elements, for any contracted Lie algebra. 

The above theorem does not say how many Lie algebras are attached to each of the 24 possible supports. 
For instance,  if $T$ is collinear to either $T_{14}$ or $T_{17}$,
there is a one-parametric family  (two-parametric   for $T=X_{(1)}$) of  non-graded isomorphic Lie algebras related to   admissible graded contractions with that support.  At the moment it is not clear when two algebras in this family (all of them 2-step solvable) are isomorphic. We aim to prove that the isomorphism classes of graded contractions
of $\Gamma_{\f{d}_4}$ and $\Gamma_{\f{d}_3}$ 
coincide with the isomorphism classes of graded contractions
 of $\Gamma_{\f{g}_2}$. The first step is to note that there exist two isomorphic contracted algebras related to supports 
 which are \textit{not} collinear; so the 
situation is more complicated than it seems.

\begin{example}\label{ex_excepcion}
The converse of Corollary~\ref{cor_pasandoporcol}  is not true for $\Gamma=\Gamma_\la$, $\la\in\{\mathfrak g_{2},\mathfrak b_{3},\mathfrak d_{4}  \}$.
 If $i\ne j$, $k\notin\ell_{ij},$ let us prove that although the nice sets:
\begin{equation}\label{eq_T10T8}
T=\{\{ i,j \},\{i,k  \},\{i, j*k \}\} \qquad\textrm{and}\qquad T'=\{\{ i,j \},\{i,k  \},\{i, i*j*k \}\}
\end{equation}
are clearly not collinear, the algebras related to $\ep^T$ and to $\ep^{T'}$ are,  in fact, isomorphic.
A way of providing an isomorphism, when $\la\in\{ \mathfrak b_{3},\mathfrak d_{4}  \},$
 is to use the  auxiliary maps $ \varphi_{rs}\in\mathrm{Aut}(\la_{r}\oplus\la_{s}\oplus\la_{r*s})$ considered
in Remark~\ref{re_fiij}.
It is a straightforward computation that the \emph{extended} map
$\theta\colon\la^{\ep^{T}}\to \la^{\ep^{T'}}$ defined by
$\theta\vert_{\la_{g_t}}=\varphi_{j*k,i*j*k}$ if $t\in\ell_{i,j*k}$ 
and by $\theta\vert_{\la_{g_t}}=\mathrm{id}$ otherwise,
is an isomorphism as required. 
\end{example}

\begin{remark}
It is worth mentioning that while \cite{Thomas1} was being written, we believed that the  Weyl group was of even greater relevance. More precisely, we thought that
 $\ep\sim\ep'$ would imply the existence of a collineation $\sigma$ with $\sigma\cdot\ep\approx\ep'$.
 If this had been true, Lemma~\ref{le_soportesiguales}  would have said that $S^{\ep'}=S^{\sigma\cdot\ep}$, and in particular the supports of $\ep$ and $\ep'$
 would have been collinear. But the latter is not necessarily true,
as seen in the above example.
 What we will at least be able to show for  $\Gamma_{\f{b}_3}$
 and $\Gamma_{\f{d}_4}$ is that if  $\ep\sim\ep'$ \emph{and}  the supports of $\ep$ and $\ep'$ are collinear, then 
 there is a collineation $\sigma$ with $\sigma\cdot\ep\approx\ep'$. 
\end{remark}

The above  need not be true in general because the proof uses some explicit isomorphisms between the contracted algebras.
As in \cite[Lemma~4.9]{Thomas1}, the next result is clear.

\begin{lemma}\label{lematitas}
Let  $\ep$ an admissible graded contraction of $\Gamma_\la$, $i\ne j\in I$, and $\theta_{ij}\colon\la^\ep\to\la^\ep$ the map given by
$\theta_{ij}\vert_{\la_{g_t}}=\varphi_{ij}$ (defined in Remark~\ref{re_fiij}) if $t\in\ell_{ij}$ 
and by $\theta\vert_{\la_{g_t}}=\mathrm{id}$ otherwise.
Then $\theta_{ij}\in\mathrm{Aut}(\la^\ep)$ if and only if:  
  for any $\{s,t\}\in \mathcal S^\ep$, either $\ell_{st}={\ell_{ij}}$ or $\ell_{st}\subset I\setminus\{i,j\}$; and $\ep_{it}=\ep_{jt}$ for all $t\ne i,j$.
\end{lemma}

This allows us to reinterpret \cite[Proposition~4.11]{Thomas1}.

\begin{prop}\label{nueva}
Let $\ep$ and $\ep'$ be admissible graded contractions of  a very good grading $\Gamma$ on a Lie algebra $\la$. Then:
\begin{enumerate} 
\item[\rm(a)] If  $\ep\sim\ep'$, there is a collineation $\sigma$ with  $\sigma\cdot\ep\approx\ep'$ except  possibly if $ \mathcal S^{\ep}=X^{(i)}$ 
or $ \mathcal S^{\ep}\subset X_{(i)}$ for some $i\in I$.
\end{enumerate}
In case $\Gamma=\Gamma_{\f{b}_3}$
 or $\Gamma_{\f{d}_4}$,
\begin{enumerate}
 \item[\rm(b)]  If  $\ep\sim\ep'$ and $ \mathcal S^{\ep}\not\sim_c \mathcal S^{\ep'}$, then there are $i,j,k$,  not forming a line, such that $ \mathcal S^{\ep}$ and $ \mathcal S^{\ep'}$ are the two nice sets in 
 Eq.~\eqref{eq_T10T8}.
\item[\rm(c)]  If  $\ep\sim\ep'$ such that  $ \mathcal S^{\ep}\sim_c \mathcal S^{\ep'}$, then there is $\sigma$ collineation with  $\sigma\cdot\ep\approx\ep'$. 
\end{enumerate}
\end{prop}

\begin{proof}
There always exists an isomorphism of graded algebras $ \varphi\colon \la^{\ep }\to\la^{\ep'}$ and   a bijection 
$\mu  \colon I\to I$ defined by $\varphi\big(\la_{g_i}\big) = \la_{g_{\mu (i)}}$.
If $\mu$ were a collineation, by composing $\varphi$ with the map $\tilde f_{\mu}\in\mathrm{Aut}(\f{d}_4)$  at the end of \S2,
 $\tilde f_{\mu}\varphi\colon \la^{\ep}\to\la^{\mu\cdot\ep'}$ would be a graded-isomorphism.
 Let $T=\mathcal S^{\ep}$. Except possibly for $T=X^{(i)}$ 
or $ T\subset X_{(i)}$, we can follow the lines of the proof of \cite[Proposition~4.11]{Thomas1}
  to find such a pair $(\varphi,\mu)$ with $\mu$ a collineation, 
  modifying the original pair when necessary.
 Whereas, if $T=X^{(i)}$ 
or $ T\subset X_{(i)}$ but $\Gamma$ is one of our gradings on $\f b_3$ or $\f{d}_4$, then the maps in Lemma~\ref{lematitas}, conveniently normalized,
 permit to find $(\varphi',\mu')$. 
 The only situation in which finding such a pair   need not be possible
 is $T=\{ \{i,j\},\{i,k\},\{i,s\} \}$, with
 $s*k\in\{i,i*j\}$. 
\end{proof}

Coming back to our purpose, the classification of the graded contractions of $ \Gamma_{\f{b}_3}$
and of  $\Gamma_{\f{d}_4}$,
 we take into account that  $\ep^{T_8}\sim \ep^{T_{10}}$, as proved in  Example~\ref{ex_excepcion}, and that the graded contractions 
$\{\ep^{T_i}\mid 1\le i\le 24,i\ne10\}$ are pairwise not equivalent
 by Proposition~\ref{nueva}. Moreover,  the number of equivalence classes with support $T_i$ is at most the number of strong equivalence classes with support $T_i$, equal to $1$ if $i\ne14,17,20$. Only for these three indices need we be concerned. However   the next result and its proof turn out to be completely analogous to \cite[Theorem~4.13]{Thomas1} for $\Gamma_{\f{g}_2}$, that is, we only need to use Proposition~\ref{nueva} to get: 

\begin{theorem}\label{th_lasclasesdeverdad}
Representatives of all the classes up to equivalence of the graded contractions of $\Gamma_{\f{b}_3}$ and $\Gamma_{\f{d}_4}$ are just:
\begin{enumerate}  
\item [\rm (i)]  $\{\eta^{T_i}\mid i\ne 8,14,17,20\}$;
\item [\rm (ii)] $\{(1, 1, 1, \lambda)\mid  \lambda\in\bb C^\times\}$ related to $T_{14}$,  where $(1, 1, 1, \lambda)\sim (1, 1, 1, \lambda')$ if and only if $ \lambda'\in\{\lambda,\frac1{\lambda}\}$;
\item [\rm (iii)]  $ \{(1, \lambda, 1, 1, \lambda)\mid  \lambda\in\bb C^\times\}$ related to $T_{17}$,
where $(1, \lambda,1, 1, \lambda)\sim (1, \lambda',1, 1, \lambda')$ if and only if $\lambda'\in\{\pm\lambda,\pm\frac1{\lambda}\}$;

\item [\rm (iv)] $\{(1, \lambda , \mu, 1, \lambda, \mu)\mid \lambda,\mu\in\bb C^\times\}$ related to $T_{20}$,
where two   maps $(1, \lambda , \mu, 1, \lambda, \mu)\sim (1, \lambda' , \mu', 1, \lambda', \mu')$ if and only if the set
$\{\pm\lambda',\pm\mu'\}$ coincides with either $\{\pm\lambda,\pm\mu\}$ or $\{\pm\frac1\lambda,\pm\frac\mu\lambda\}$ or $\{\pm\frac\lambda\mu,\pm\frac1\mu\}$.
\end{enumerate} 
\end{theorem}

\begin{conclusion}\label{conclu}
Our main result is  Theorem~\ref{th_lasclasesdeverdad}, which provides the desired classification of graded contractions up to equivalence of both $\Gamma_{\f{b}_3}$ and $\Gamma_{\f{d}_4}$. 
Then, for each $\la\in\{\f{b}_3,\f{d}_4 \}$, there exist 20 not isomorphic Lie algebras  jointly with three infinite families. Although some of their properties are exhibited in Theorem~\ref{teo_lasalgebras}, it is clear that these properties alone do not provide enough invariants to distinguish the different isomorphism classes. 

It has been a surprise for us to be able to apply the results of $\f g_2$ to other algebras almost without any change, thus producing new solvable and nilpotent bigger algebras. The fact that the results are essentially identical does not mean that the proofs are identical, the work of adaptation being non-trivial (done mainly in Lemmas~\ref{le_corchete}, \ref{co_lasnuestrassongood}, \ref{le_lineacopiassl2}, and Theorem~\ref{teo_fuerte}, jointly with a new writing of the results which provides a less technical approach).
\end{conclusion}

 	
\section{Comments on the real case  }	\label{se_real}

Now let $\mathbb{O}$ be the real algebra  with basis $\{e_0,\dots, e_7\}$ and products $e_0=1$, $e_i^2=-1$ for all $i\ne0$,  and $e_ie_j=-e_je_i=e_k$,   whenever 
$(i,j,k)$ or any of its cyclic permutations belongs to $\mathbf{L}$. 
Thus  $\mathbb{O}$ 
is a division alternative algebra and the polar form of $n$
(we abuse the   notation by using the same $n$ as before)  is a scalar product.
We   denote  by $\f{g}_{2,-14}\equiv \Der(\mathbb{O})$, which is now  a compact Lie algebra with a negative definite Killing form,  
$\f{b}_{3,-21}\equiv \{f\in \mathfrak{so}(\mathbb{O},n):f(1)=0\}\subset\mathfrak{d}_{4,-28}\equiv \mathfrak{so}(\mathbb{O},n)$, which are isomorphic to $ \mathfrak{so}(7,\bb R)$ and $ \mathfrak{so}(8,\bb R)$ respectively, both compact Lie algebras. (The second subindex refers to the signature of the Killing form, as usual.)
Again the three Lie algebras $\f{g}_{2,-14}$, $\f{b}_{3,-21}$ and $\f{d}_{4,-28}$ have very good $\bb Z_2^3$-gradings induced from the $\bb Z_2^3$-grading on $\mathbb{O}$ given by $\mathbb{O}_{g_i}=\mathbb R e_i$ and $\mathbb{O}_{g_0}=\mathbb R 1$. 
The arguments in \S\ref{se_resultadosengrcont} imply several facts: if $\ep$ is a graded contraction of $\la$, then there exists $\ep'$ an admissible graded contraction such that  $\la^\ep$ is  isomorphic to $\la^{\ep'}$. Furthermore, the support of $\ep'$ is a nice set $T$ which is collinear to one of the 24 nice sets in Theorem~\ref{teo_lasalgebras}. As the Weyl group is 
again the whole group $\Aut(\bb Z_2^3)$, $\la^{\ep'}$ is isomorphic to  $\la^{\ep''}$ for $\ep''$ another  admissible graded contraction with support equal to one of the nice sets in Theorem~\ref{teo_lasalgebras}.  Moreover, for each   nice set $T$ in that theorem, $\ep^T$ is a graded contraction and $\la^{\ep^T}$ is a Lie algebra satisfying  the properties described in  Theorem~\ref{teo_lasalgebras}. 
The only difference between the real case and the complex one  is that now there are more equivalence classes of nonisomorphic Lie algebras. The 
Lie algebras in Theorem~\ref{th_lasclasesdeverdad},
 both the 20 Lie algebras $\la^{\ep^T}$ and the three infinite families $\la^{\eta^\lambda}$, $\la^{\mu^\lambda}$ and $\la^{\beta^{\lambda,\lambda'}}$ related to the nice sets  in \eqref{eq_malos}, are still not equivalent, but they do not constitute the whole set of equivalence classes.  
   The following proposition provides a remarkable example.
  
 \begin{prop}
 For $\ell\in\mathbf{L}$, define $\ep\colon G\times G\to \bb R$ by  
 $\ep(g_i,g_j)=0$ if either $i=j$ or $i=0$ or $j=0$, $\ep(g_i,g_j)=-1$ if $i\ne j$, $i,j\notin\ell$, and $\ep(g_i,g_j)=1$ otherwise. 
 Then $\ep$ is a graded contraction and the Lie algebra $\la^\ep$ is not isomorphic to  $\la$ (for any $\la\in\{\f{g}_{2,-14},\f{b}_{3,-21},\f{d}_{4,-28}\}$).
 \end{prop}
 
 That is, $\ep$ provides an admissible graded contraction with support   $X$  such that $\la^\ep$ is isomorphic to the split real algebra of type $G_2, B_3$ and $D_4$ respectively,   while  for $\ep$ constantly equal to $1$, the support is also $X$ and $\la^\ep $  is compact.
 As far as we know, there are no precedents in the literature on a complete classification of graded contractions of a real Lie algebra.   Future work could go in this direction.

\end{document}